%% file: 0_main.tex
\documentclass[preprint,11pt]{elsarticle}

\usepackage{amssymb}
\usepackage{amsmath}
\usepackage{amsthm}
\usepackage{enumitem}
\usepackage{setspace}
\usepackage{subcaption}  
\usepackage{xcolor}
\usepackage[top=1in, bottom=1in, left=1in, right=1in]{geometry}

\journal{Journal of Computational Physics}

\definecolor{revgreen}{rgb}{0,0.5,0}

\usepackage{cases}
\usepackage{algorithm}
\usepackage{algpseudocode}
\algdef{SE}[SUBALG]{Indent}{EndIndent}{}{\algorithmicend\ }%
\algtext*{Indent}
\algtext*{EndIndent}

\theoremstyle{plain}
\newtheorem{theorem}{Theorem}[section]
\newtheorem{lemma}[theorem]{Lemma}
\newtheorem{proposition}[theorem]{Proposition}

\theoremstyle{definition}

\newtheorem{assumption}[theorem]{Assumption}

\theoremstyle{remark}
\newtheorem{remark}[theorem]{Remark}

\input{my_commands}

\begin{document}

\begin{frontmatter}



\title{IEnSF: Iterative Ensemble Score Filter for Reducing Error in Posterior Score Estimation in Nonlinear Data Assimilation}


\author[label3]{Zezhong Zhang} 
\author[label2]{Feng Bao} 
\author[label1]{Guannan Zhang\corref{cor1}} 

\cortext[cor1]{Corresponding author}

\affiliation[label3]{organization={Department of Mathematics and Statistics, Auburn University},
            addressline={1550 E. Glenn Ave.}, 
            city={Auburn},
            postcode={36830}, 
            state={AL},
            country={USA}}

\affiliation[label2]{organization={Department of Mathematics, Florida State University},
            addressline={1017 Academic Way}, 
            city={Tallahassee},
            postcode={32306}, 
            state={FL},
            country={USA}}
            
\affiliation[label1]{organization={Computer Science and Mathematics Division, Oak Ridge National Laboratory},
            addressline={1 Bethel Valley Road}, 
            city={Oak Ridge},
            postcode={37831}, 
            state={TN},
            country={USA}}

\begin{abstract}
 {The Ensemble Score Filter (EnSF) is a score-based diffusion model approach for solving high-dimensional and nonlinear data assimilation problems.}
While initial applications of EnSF to the Lorenz-96 model and the quasi-geostrophic system showed potential, the current method employs a heuristic weighted sum to combine the prior and the likelihood score functions. This introduces a structural error into the estimation of the posterior score function in the nonlinear setting. This work addresses this challenge by developing an iterative ensemble score filter (IEnSF) that applies an iterative algorithm as an outer loop around the reverse-time stochastic differential equation solver. When the state dynamics or the observation operator is nonlinear, the iterative algorithm can gradually reduce the posterior score estimation error by improving the accuracy of approximating the conditional expectation of the likelihood score function. 
 {The number of iterations required depends on the distance between the prior and posterior distributions.}
Numerical experiments demonstrate that the IEnSF algorithm substantially reduces the error in posterior score estimation in the nonlinear setting and thus improves the accuracy of tracking high-dimensional dynamical systems.

\end{abstract}




\tnotetext[fn1]{{\bf Notice}:  This manuscript has been authored by UT-Battelle, LLC, under contract DE-AC05-00OR22725 with the US Department of Energy (DOE). The US government retains and the publisher, by accepting the article for publication, acknowledges that the US government retains a nonexclusive, paid-up, irrevocable, worldwide license to publish or reproduce the published form of this manuscript, or allow others to do so, for US government purposes. DOE will provide public access to these results of federally sponsored research in accordance with the DOE Public Access Plan.}

\end{frontmatter}


\input{1_intro}

\input{2_background}

\input{3_score_derivation}

\input{4_score_approximation}
\input{5_numerical_new}

\input{6_conclusoin}

\section*{Acknowledgment}
This material is based upon work supported by the U.S. Department of Energy, Office of Science, Office of Advanced Scientific Computing Research, Applied Mathematics program under the contract ERKJ443 at the Oak Ridge National Laboratory, which is operated by UT-Battelle, LLC, for the U.S. Department of Energy under Contract DE-AC05-00OR22725.

\bibliographystyle{elsarticle-num} 
\bibliography{ref_zzz}

\appendix
\input{9_appendix}


\appendix
\end{document}

%% file: my_commands.tex
\newcommand{\mF}{\mathcal{F}}

\newcommand{\mM}{\mathcal{M}}
\newcommand{\mN}{\mathcal{N}}
\newcommand{\mD}{\mathcal{D}}

\newcommand{\bbE}{{\mathbb{E}}}

\newcommand{\bbR}{\mathbb{R}}

\newcommand{\bI}{\mathbf{I}}

\newcommand{\bQ}{\boldsymbol{Q}}
\newcommand{\bL}{\boldsymbol{L}}

\newcommand{\bJ}{\mathbf{J}}

\newcommand{\bZ}{\boldsymbol{Z}}
\newcommand{\bz}{\boldsymbol{z}}
\newcommand{\bX}{\boldsymbol{X}}
\newcommand{\bx}{\boldsymbol{x}}
\newcommand{\bY}{\boldsymbol{Y}}
\newcommand{\by}{\boldsymbol{y}}

\newcommand{\bS}{\boldsymbol{S}}
\newcommand{\bSigma}{\boldsymbol{\Sigma}}
\newcommand{\bsigma}{\boldsymbol{\sigma}}
\newcommand{\bmu}{\boldsymbol{\mu}}
\newcommand{\bveps}{\boldsymbol{\varepsilon}}
\newcommand{\bH}{\textbf{H}}

\newcommand{\bzero}{\boldsymbol{0}}
\newcommand{\bW}{\boldsymbol{W}}

\newcommand{\ba}{\boldsymbol{a}}

\newcommand{\deq}{\stackrel{\text{d}}{=}}

\makeatletter
\DeclareRobustCommand{\cev}[1]{%
  \mathpalette\do@cev{#1}%
}
\newcommand{\do@cev}[2]{%
  \fix@cev{#1}{+}%
  \reflectbox{$\m@th#1\vec{\reflectbox{$\fix@cev{#1}{-}\m@th#1#2\fix@cev{#1}{+}$}}$}%
  \fix@cev{#1}{-}%
}
\newcommand{\fix@cev}[2]{%
  \ifx#1\displaystyle
    \mkern#23mu
  \else
    \ifx#1\textstyle
      \mkern#23mu
    \else
      \ifx#1\scriptstyle
        \mkern#22mu
      \else
        \mkern#22mu
      \fi
    \fi
  \fi
}

%% file: 1_intro.tex
\section{Introduction}



Data assimilation (DA) is the process of integrating noisy and partial observational data into complex dynamical systems, which are often high-dimensional, nonlinear, and potentially chaotic, to obtain improved estimates of the system state. It plays a central role in a wide range of scientific and engineering applications, including weather forecasting, climate modeling, material sciences, and biology \cite{evensen1994sequential, houtekamer1998data, shaman2012forecasting}. Nonlinear filtering methods were developed to address the data assimilation problem through a sequential Bayesian inference framework \cite{kalman1960new,evensen2009data,jazwinski2013stochastic, gordon1993novel}. To account for uncertainties arising from both the system dynamics and the observational noise, nonlinear filtering aims to characterize the conditional distribution of the state variable given the observational data, known as the \emph{filtering distribution}. From this distribution, one can obtain not only the state estimate but also an associated measure of uncertainty.

Depending on the assumptions imposed on the filtering distribution, two major classes of methods have been developed: the \emph{Kalman family} \cite{kalman1960new, gelb1974applied, julier1997new, evensen1994sequential, houtekamer1998data, hunt2007efficient} and the \emph{Particle family} \cite{gordon1993novel, kitagawa1996monte,doucet2001sequential}. Under the Gaussian assumption, which holds when both the state dynamics and the observation operator are linear and the noise is Gaussian, the Kalman filter and its extensions provide analytical updates for the mean and covariance of the Gaussian filtering distribution. 
However, maintaining and updating the covariance matrix becomes computationally prohibitive for high-dimensional systems. To address this, ensemble-based methods such as the \emph{Ensemble Kalman Filter} (EnKF) \cite{evensen1994sequential} and the \emph{Local Ensemble Transform Kalman Filter} (LETKF) \cite{hunt2007efficient} approximate the state distribution using a finite ensemble of samples, replacing the exact covariance with the sample covariance, thereby enabling scalability to large systems. 
 {A key ingredient enabling this when the ensemble is much smaller than the state dimension is covariance localization, which suppresses spurious long-range sample correlations, either by tapering the sample covariance or by restricting each local analysis to nearby observations, and is typically combined with multiplicative inflation to counteract the underestimation of ensemble variance.}
In contrast, particle-family filters represent the distribution as a weighted empirical distribution, allowing for non-Gaussian and multi-modal filtering distributions that are better suited for nonlinear problems. 
While particle filters perform well in low-dimensional nonlinear systems, their effectiveness in high-dimensional settings is severely limited by weight degeneracy and the curse of dimensionality, as the number of required particles grows exponentially with system dimension and observation dimension  {\cite{bengtsson2008curse, snyder2008obstacles, snyder2015performance}}. 
Although several improved particle filters have been proposed \cite{pitt1999filtering, gilks2001following, Rebeschini2013CanLP} {, including localized particle filters designed for high-dimensional systems \cite{poterjoy2016localized, poterjoy2016efficient}},  {their broad adoption in large-scale operational applications remains limited}.
As a result, ensemble Kalman methods such as LETKF remain the practical standard due to their scalability, despite their inherent Gaussian assumption, which introduces bias in strongly nonlinear regimes and requires careful tuning of localization and inflation parameters. 
 {To better handle nonlinearity, a family of iterated Kalman and ensemble-variational methods has been developed, including the iterated extended Kalman filter \cite{bell1993iterated}, the iterative ensemble Kalman filter and smoother \cite{sakov2012iterative, bocquet2014iterative}, and the Levenberg--Marquardt form of the iterative ensemble smoother \cite{chen2013levenberg}. These methods handle nonlinearity through an outer loop that repeatedly relinearizes the observation operator and refines a Gaussian analysis, using the observation-operator gradient obtained either directly or from the ensemble.}
 {Even with these advances,} the greatest practical challenge in modern data assimilation lies in the simultaneous presence of high dimensionality and nonlinearity, which leads to high-dimensional non-Gaussian filtering distributions that render both Kalman- and particle-type methods either  {statistically inconsistent or computationally infeasible.}

To address this combined challenge, recent advances have drawn inspiration from \emph{generative AI}, particularly from \emph{diffusion models}, which have demonstrated remarkable success in representing complex, high-dimensional, and non-Gaussian data distributions such as images and videos \cite{kingma2021variational, song2020score, ho2020denoising, ho2022video, karras2022elucidating}.  {These models are also capable of conditional sampling, which has been used effectively for Bayesian inference tasks \cite{kawar2022denoising, chung2022diffusion, zhang2025exact}}. In the context of nonlinear filtering, the \emph{score-based filtering framework} \cite{bao2023scorebased} adapts diffusion models to describe the evolution of filtering distributions governed by system dynamics and Bayesian updates. 
The central idea is to represent the filtering distribution through its score function, defined as the gradient of the log-density, and to generate posterior samples by solving a reverse-time stochastic differential equation (SDE). 
While the neural score-based approach effectively handles nonlinear and non-Gaussian systems, it requires retraining a neural network to learn the score function at every assimilation step, making it computationally impractical for large-scale applications {~\cite{hodyss2026using}}.
To overcome this limitation, Bao et al.~\cite{ensf_cmame} proposed the \emph{Ensemble Score Filter} (EnSF), a training-free alternative that estimates the score function directly from ensemble samples. This approach preserves the expressive power of diffusion-based models while achieving the computational efficiency necessary for high-dimensional nonlinear filtering, and has been successfully applied to several challenging large-scale data assimilation and uncertainty quantification problems \cite{lu2024diffusion,  zhang2025exact, bao2025nonlinear, liang2025ensemble}. 
However, EnSF still exhibits a structural bias in its posterior score estimation, since its heuristic formulation enforces consistency with the likelihood score only at the initial and terminal times of the reverse-time SDE. Consequently, the posterior samples generated by the reverse-time SDE may deviate from the true posterior, especially in cases with sparse observations where state updates are insufficiently informed for unobserved components.

In this work, we address the limitations of EnSF by introducing the \emph{Iterative Ensemble Score Filter} (IEnSF).
The central idea of IEnSF is to iteratively refine an approximate posterior score that is derived analytically under a Gaussian mixture assumption on the prior distribution.
The refinement is achieved by repeatedly solving the reverse-time SDE, using the resulting posterior samples to update the estimated posterior score, thereby improving its accuracy.
This iterative process effectively mitigates structural errors in the posterior score and enhances the overall consistency of the posterior distribution. Through comprehensive numerical experiments, we demonstrate that IEnSF achieves significant improvements in posterior consistency and bias reduction.
In particular, we show reduced KL divergence in low-dimensional linear problems and superior performance compared to state-of-the-art methods such as LETKF in high-dimensional, nonlinear Lorenz-96 systems with sparse observations.
 {IEnSF shares the outer-loop structure and the use of the observation-operator gradient with these iterated Kalman and ensemble-variational methods, but differs in what the outer loop refines: it refines only a reference that places the evaluation point of the likelihood score, while the posterior ensemble is generated by the reverse SDE without imposing a Gaussian form.}
The main contributions of this work are summarized as follows:
\begin{itemize}[leftmargin=15pt]
\item We derive an exact analytical expression for the posterior score function under a Gaussian mixture prior assumption, which serves as the theoretical foundation of the proposed algorithm.
\item We develop efficient numerical approximations of the exact posterior score function based on linearization of the observation operator.

\item We introduce an iterative procedure to refine the likelihood score estimation during reverse SDE sampling, which substantially reduces the error in posterior approximation.

\end{itemize}

The rest of the paper is organized as follows.  
Section~\ref{sec.background} provides background on the EnSF framework and diffusion-based sampling.  
In Section~\ref{sec.iensf}, we introduce the proposed IEnSF method: Section~\ref{sec.exact.post.score} derives the exact posterior score under Gaussian mixture priors, Section~\ref{sec.approx.post.score} presents practical approximations, and Section~\ref{sec.iensf.sum} summarizes the full algorithm with iterative refinement.  
Numerical results are presented in Section~\ref{sec.numerical}, and conclusions are drawn in Section~\ref{sec.conclusion}.

%% file: 2_background.tex
\section{Background}\label{sec.background}

We briefly recall the mathematical formulation of the data assimilation (DA) problem and introduce the EnSF framework.

\subsection{Formulation of the data assimilation problem}

Let $\bX_n \in \bbR^d$ denote the system state at discrete time $t_n$, where $\{t_n\}_{n=0}^N$ represents the sequence of observation times.  
The evolution of the state from $t_n$ to $t_{n+1}$ is governed by a stochastic dynamical model,
\begin{equation}\label{eq.filtering.state}
    \text{\em State model:}\quad  
    \bX_{n+1} = \mF_n(\bX_n, \omega_n),
\end{equation}
where $\mF_n: \bbR^d \times \bbR^k \to \bbR^d$ is a numerical forecast model that propagates the state from $t_n$ to $t_{n+1}$, and $\omega_n \in \bbR^k$ represents either the intrinsic stochasticity of the state dynamics or the accumulated uncertainty introduced by the numerical solver.  
Note that the time interval between $t_n$ and $t_{n+1}$ may correspond to multiple internal integration steps of the numerical solver used to advance the forward model.

Observations $\bY_{n+1} \in \bbR^r$ collected at time $t_{n+1}$ are related to the true state $\bX_{n+1}$ through
\begin{equation}\label{eq.filtering.obs}
    \text{\em Observation model:}\quad  
    \bY_{n+1} = \mM(\bX_{n+1}) + \bveps_{n+1},
\end{equation}
where $\mM: \bbR^d \to \bbR^r$ is the observation operator, and $\bveps_{n+1}$ denotes the observational noise at time $t_{n+1}$.  
The objective of data assimilation is to estimate the posterior( or filtering) distribution of the current state $\bX_{n+1}$ conditioned on all available observations $\bY_{1:n+1} := \{\bY_1, \ldots, \bY_{n+1}\}$, which is denoted by $p_{\bX_{n+1}|\bY_{1:n+1}}(\bx_{n+1}|\by_{1:n+1})$.

To obtain this filtering distribution, a recursive Bayesian inference framework is typically employed, in which the existing filtering distribution is sequentially propagated and updated as new observations become available.  
The update from $p_{\bX_n|\bY_{1:n}}$ (at $t_n$) to $p_{\bX_{n+1}|\bY_{1:n+1}}$ (at $t_{n+1}$) consists of two main steps:

\begin{itemize}[leftmargin=20pt]\itemsep0.2cm
    \item \textbf{Prediction step.}  
    The prior distribution at $t_{n+1}$ is obtained by propagating the existing posterior $p_{\bX_{n}|\bY_{1:n}}(\bx_{n}|\by_{1:n})$ at $t_n$, through the dynamical model in Eq.~\eqref{eq.filtering.state} and its density is given by the Chapman--Kolmogorov equation:
    \begin{equation}\label{eq.filtering.prior.ck}
        p_{\bX_{n+1}|\bY_{1:n}}(\bx_{n+1}|\by_{1:n}) 
        = \int p_{\bX_{n+1}|\bX_n}(\bx_{n+1}|\bx_n) \,
          p_{\bX_n|\bY_{1:n}}(\bx_n|\by_{1:n}) \, d\bx_n,
    \end{equation}
    where $p_{\bX_{n+1}|\bX_n}(\bx_{n+1}|\bx_n)$ is the transition density derived from Eq.~\eqref{eq.filtering.state}.  
    The resulting distribution $p_{\bX_{n+1}|\bY_{1:n}}$ is referred to as the \emph{prior filtering distribution} for the next Bayesian update.  
    In practice, instead of solving Eq.~\eqref{eq.filtering.prior.ck} directly, the prior distribution is obtained by propagating an ensemble of samples from $p_{\bX_{n}|\bY_{1:n}}$ through the dynamical model in Eq.~\eqref{eq.filtering.state}.  
    The resulting prior ensemble representing $p_{\bX_{n+1}|\bY_{1:n}}$ is denoted by
    \begin{equation}\label{eq.prior.sample}
        \mathcal{D}_{n+1}^{\text{prior}} := 
        \{\bx_{n+1|n}^1, \ldots, \bx_{n+1|n}^K\}.
    \end{equation}

    \item \textbf{Update step.}  
    Upon receiving a new observation $\bY_{n+1}$, the prior distribution $p_{\bX_{n}|\bY_{1:n}}(\bx_{n}|\by_{1:n})$ at $t_n$ is updated via Bayes’ theorem to obtain the posterior:
    \begin{equation}\label{eq.filtering.bayesian.update}
        \underbrace{p_{\bX_{n+1}|\bY_{1:n+1}}(\bx_{n+1} | \by_{1:n+1})}_{\text{Posterior}}
        \propto
        \underbrace{p_{\bX_{n+1}|\bY_{1:n}}(\bx_{n+1} | \by_{1:n})}_{\text{Prior}}
        \underbrace{p_{\bY_{n+1}|\bX_{n+1}}(\by_{n+1} | \bx_{n+1})}_{\text{Likelihood}},
    \end{equation}
    where the likelihood $p_{\bY_{n+1}|\bX_{n+1}}(\by_{n+1} | \bx_{n+1})$ is defined by the observation model in Eq.~\eqref{eq.filtering.obs}.  
    The resulting posterior ensemble is denoted as
    \begin{equation}\label{eq.post.sample}
        \mathcal{D}_{n+1}^{\text{posterior}} := 
        \{\bx_{n+1|n+1}^1, \ldots, \bx_{n+1|n+1}^K\},
    \end{equation}
    which serves as the initial ensemble for the next assimilation cycle.
\end{itemize}

\subsection{Overview of the Ensemble Score Filter (EnSF)}\label{sec.ensf.overview}

We next provide a brief overview of the Ensemble Score Filter (EnSF) introduced in \cite{ensf_cmame,bao2023scorebased} for solving the data assimilation problem described above, with an emphasis on two major drawbacks of its current score estimation scheme.  
The core idea of EnSF is to employ score-based diffusion models to represent and propagate the dynamics of the posterior filtering distribution in Eq.~\eqref{eq.filtering.bayesian.update}.  
A crucial component of EnSF is therefore the numerical approximation of the score function, which must be carried out consistently in both the prediction and update steps of the recursive Bayesian framework.

\subsubsection{Score-based diffusion model}
We begin with a brief introduction to score-based diffusion models, which form the backbone of the EnSF framework for representing filtering distributions.  
Diffusion models are a class of generative models capable of sampling from a target distribution $p_{\bX}$.  
In EnSF, these target distributions correspond to the prior and posterior filtering distributions at each assimilation step.

The diffusion model consists of two coupled processes: the \emph{forward process} and the \emph{reverse process}.  
The forward process gradually diffuses the target distribution $p_{\bX}$ into pure Gaussian noise by continuously adding small Gaussian perturbations, whereas the reverse process reconstructs samples from the target distribution by progressively denoising from pure Gaussian noise.  
In EnSF, both processes are defined continuously over a pseudo-time domain $t \in \mathcal{T} = (0, T)$.  
In this continuous-time setting, both processes are governed by stochastic differential equations (SDEs) or ODEs .  
The \emph{forward SDE} describes the forward diffusion process as follows:
\begin{equation}\label{eq.forward.sde}
    d\bZ_{t} = b(t) \bZ_{t}\, dt + \sigma(t)\, d\bW_t, \quad \bZ_{0} \deq \bX
\end{equation}
where $\bW_t$ is the standard Brownian motion running forward in time, and the initial distribution $\bZ_{0}$ follows the target distribution $p_{\bX}$.
According to \cite{kingma2021variational}, the drift and diffusion coefficients $b(t)$ and $\sigma(t)$ in Eq.~\eqref{eq.forward.sde} are determined by the noise schedule $(\alpha_t, \beta_t)$ as
\begin{equation}\label{eq.forward.sde.coef}
b(t) = \frac{{\rm d} \log \alpha_t}{{\rm d} t},\;\;\; \sigma^2(t) = \frac{{\rm d} \beta_t^2}{{\rm d}t} - 2 \frac{{\rm d}\log \alpha_t}{{\rm d}t} \beta_t^2. 
\end{equation}
The noise schedule $(\alpha_t, \beta_t)$ controls the diffusion rate through the conditional transition distribution
\begin{equation}\label{eq.forward.conditional}
    p_{\bZ_t | \bZ_0}(\bz_t | \bz_0) \deq \mN(\alpha_t \bz_0, \beta_t^2 \bI).
\end{equation}
By enforcing the signal-to-noise ratio (SNR) $\alpha_t^2 / \beta_t^2$ to decrease monotonically to zero as $t \to T$, 
the terminal state $\bZ_T$ converges to a pure Gaussian noise whose distribution is independent of the initial state $\bZ_0$.

Let $\bS(\bz_t, t) := \nabla_{\bz_t} \log p_{\bZ_t}(\bz_t)$ denote the score function of the forward SDE.  
Given this score, samples from the target distribution $p_{\bX}$ can be generated by solving the following \textit{reverse-time SDE} for $\cev{\bZ}_0$:
\begin{equation}\label{eq.reverse.sde}
    d\cev{\bZ}_t = \left[ b(t) \cev{\bZ}_t - \sigma^2(t)\, \bS(\cev{\bZ}_t, t)\right] dt + \sigma(t)\, d\cev{\bW}_t,
\end{equation}
where $\cev{\bW}_t$ is a Brownian motion running backward in time, and the terminal condition $\cev{\bZ}_T$ follows a Gaussian determined by the noise schedule. 
It has been shown that the forward and reverse processes share the same marginal distribution, i.e., $\cev{\bZ}_t \deq \bZ_t$ \cite{anderson1982reverse}.  
Consequently, solving the reverse SDE from $\cev{\bZ}_T$ to $\cev{\bZ}_0 \deq \bZ_0 \sim p_{\bX}$ produces samples from the target distribution. 
Hence, once the score function is known, the reverse SDE provides a principled mechanism for sampling from arbitrary target distributions.  
This mechanism forms the foundation of the EnSF, which leverages score functions to represent and propagate filtering distributions.

\subsubsection{Prior score estimation in EnSF}\label{sec.ensf.prior}
We now describe how the score-based diffusion model is applied to represent the prior filtering distribution $p_{\bX_{n+1} | \bY_{1:n}}$ in Eq.~\eqref{eq.filtering.prior.ck} in the prediction step.  
The corresponding forward diffusion process is defined as:
\begin{equation}\label{eq.prior.forward.sde}
    d\bZ_{n+1,t} = b(t)\, \bZ_{n+1,t}\, dt + \sigma(t)\, dW_t, \quad \bZ_{n+1,0} \deq \bX_{n+1} | \bY_{1:n}
\end{equation}
where the initial distribution $\bZ_{n+1,0}$ follow the prior filtering distribution $p_{\bX_{n+1} | \bY_{1:n}}$, and the subscript $(n+1)$ indicates that the process corresponds to the $(n+1)$-th assimilation step.  
Following \cite{bao2023scorebased, ensf_cmame}, we adopt the noise schedule $\alpha_t = 1 - t$, $\beta_t^2 = t$, and $T = 1$, which ensures that the terminal state $\bZ_{n+1,1}$ converges to a standard Gaussian distribution $\mathcal{N}(\bzero, \bI)$, independent of the initial condition. The score function of this forward SDE is defined as
\begin{equation}
    \bS_{n+1|n}(\bz_{n+1,t}, t) := 
    \nabla_{\bz_{n+1,t}} \log p_{\bZ_{n+1,t}}(\bz_{n+1,t}),
\end{equation}
which encode the generative power of the prior distribution $p_{\bX_{n+1} | \bY_{1:n}}$.  
By substituting this score function into the reverse SDE in Eq.~\eqref{eq.reverse.sde} and solving it backward in time from the standard Gaussian terminal condition, one can generate samples that approximate the prior filtering distribution.

Because the filtering distribution evolves dynamically over time, learning-based score estimation approaches (e.g., training a neural network (NN) to approximate the score function) are computationally prohibitive due to the need for frequent retraining at each assimilation step \cite{bao2023scorebased}.  
Instead, the EnSF method derives an explicit analytical expression for the prior score function:
\begin{equation}\label{eq.ensf.prior.derive}
\bS_{n+1|n}(\bz_{n+1,t}, t) = \int  - \frac{\bz_{n+1,t}- \alpha_t \bz_{n+1,0}}{\beta^2_t} w(\bz_{n+1,t},  \bz_{n+1,0})  p_{\bZ_{n+1,0}}(\bz_{n+1,0}) d \bz_{n+1,0},\\
\end{equation}
%
where the weight function $w(\bz_{n+1,t},  \bz_{n+1,0})$ is defined as
\begin{equation}\label{eq.ensf.weight}
w(\bz_{n+1,t},  \bz_{n+1,0}) := \frac{ p_{\bZ_{n+1,t}| \bZ_{n+1,0}}(\bz_{n+1,t} | \bz_{n+1,0}) }{\int p_{\bZ_{n+1,t}| \bZ_{n+1,0}}(\bz_{n+1,t} | \bz'_{n+1,0})  p_{\bZ_{n+1,0}}(
\bz'_{n+1,0}) d \bz'_{n+1,0}},
\end{equation}
satisfying satisfies the normalization condition $\int w(\bz_{n+1,t},  \bz_{n+1,0}) p_{\bZ_{n+1,0}}(\bz_{n+1,0}) d \bz_{n+1,0} = 1$. 
This formulation exploits the Gaussian property of 
$p_{\bZ_{n+1,t} | \bZ_{n+1,0}}(\bz_{n+1,t} | \bz_{n+1,0}) \sim \mN(\alpha_t \bZ_{n+1,0}, \beta_t^2 \bI_d)$ \cite{ensf_cmame}.  
Since the integral in Eq.~\eqref{eq.ensf.prior.derive} is taken with respect to the initial distribution $p_{\bZ_{n+1,0}}$, which is also the prior filtering distribution $p_{\bX_{n+1} | \bY_{1:n}}$, 
the prior score can be efficiently approximated using Monte Carlo integration based on the prior ensemble samples in Eq.~\eqref{eq.prior.sample}.

\subsubsection{Posterior score estimation in EnSF}\label{sec.ensf.post}
After receiving the new observation $\bY_{n+1}$, the prior score $\bS_{n+1|n}$ in Eq.~\eqref{eq.ensf.prior.derive} is updated by incorporating the likelihood information according to Bayes’ rule in Eq.~\eqref{eq.filtering.bayesian.update}, yielding an approximation of the posterior score function $\bS_{n+1|n+1}$.  
Specifically, EnSF combines the estimated prior score with the gradient of the log-likelihood (likelihood score) as follows:
\begin{equation}\label{eq:ensf.posterior.score}
    \bS_{n+1|1:n+1}(\bz,t) \;\approx\; \hat{\bS}_{n+1|1:n+1}(\bz,t) 
    = \hat{\bS}_{n+1|1:n}(\bz,t) + h(t)\,\nabla_{\bz} \log p_{\bY_{n+1} | \bX_{n+1}}(\by_{n+1} | \bz),
\end{equation}
where $\hat{\bS}_{n+1|n}$ denotes the Monte Carlo estimator of the prior score from Eq.~\eqref{eq.ensf.prior.derive},  
$p_{\bY_{n+1} | \bX_{n+1}}(\by_{n+1} | \cdot)$ is the likelihood function defined by the observation model in Eq.~\eqref{eq.filtering.obs},  
and $h: [0,1] \to \bbR$ is a continuous time-damping function that modulates the influence of the observation during the reverse SDE sampling process.  
In the current EnSF implementation, $h(t)$ is chosen as a monotonically decreasing function satisfying the boundary conditions $h(0)=1$ and $h(1)=0$ (for example, $h(t)=1-t$).  
This choice ensures that the contribution of the likelihood score is gradually introduced as the pseudo-time $t$ decreases from $1$ to $0$, allowing the reverse SDE to progressively transform the prior distribution into the posterior distribution.

Although EnSF has demonstrated encouraging results in high-dimensional nonlinear systems such as the Lorenz–96 and quasi-geostrophic models \cite{ensf_cmame, bao2025nonlinear, liang2025ensemble}, its current formulation for posterior score estimation in Eq.~\eqref{eq:ensf.posterior.score} suffers from two major limitations:

\begin{itemize}[leftmargin=30pt]\itemsep0.2cm
    \item[(\bf C1)] \textit{Structural errors in the posterior score estimator.}  
    The posterior score estimator $\hat{\bS}_{n+1|n+1}$ is only consistent with the true posterior score at $t=0$ and $t=1$, leading to structural bias for intermediate $t$ values.  
     {Because this bias lies in the posterior score expression itself rather than in the numerical solution of the reverse SDE, it does not vanish as the ensemble size grows or the reverse SDE solver is refined.}
    Moreover, the Monte Carlo approximation introduces additional stochastic error that compounds this structural bias, resulting in systematic deviation of the posterior ensemble $\mathcal{D}_{n+1}^{\rm posterior}$ from the true posterior distribution.
    \item[(\bf C2)] \textit{Insufficient updates for unobservable state components.}  
    When the observation operator $\mM$ involves only a subset of the state variables, the data likelihood score $\nabla_{\bz}\log p_{\bY_{n+1}|\bX_{n+1}}(\bz)$ vanishes along the unobserved directions.  
    As a result, the likelihood term provides no direct update for these components, leaving the corresponding posterior estimates largely unchanged from their prior values.
\end{itemize}

%% file: 3_score_derivation.tex
\section{The Iterative Ensemble Score Filter (IEnSF)}\label{sec.iensf}

In this section, we introduce the Iterative Ensemble Score Filter (IEnSF), developed to address challenges {\bf (C1)} and {\bf (C2)} identified in Section~\ref{sec.ensf.post}. The objective is to construct a more accurate posterior score approximation. We begin in Section~\ref{sec.exact.post.score} by deriving an exact expression for the posterior score under the assumption of a Gaussian mixture (GM) prior. Since this exact expression still contains intractable terms, we propose in Section~\ref{sec.approx.post.score} an iterative approximation scheme to make the posterior score computationally feasible. Finally, in Section~\ref{sec.iensf.sum}, we integrate these developments into the DA problem and provide a complete summary of the IEnSF framework for DA.

For notational simplicity, in this section, we derive the posterior score expression and its approximation by focusing on a single Bayesian update step. Accordingly, we drop the data assimilation time index $n$ and frame the Bayesian problem in a static setting. Specifically, we consider:
\begin{align}
    \text{Prior: } & \bX := \bX_{n+1} | \bY_{1:n}, \label{eq.bayes.static.prior} \\
    \text{Observation: } & \bY := \bY_{n+1} = \mM(\bX) + \bveps, \label{eq.bayes.static.obs}
\end{align}
where $\bveps \sim \mN(\bzero, \bSigma_{\text{obs}})$ and $\bX \sim \mD^{\text{prior}} = \{ \bx_k \}_{k=1}^K$ are the prior samples. Under this Gaussian observation noise, the likelihood is
\begin{equation}\label{eq.obs.likelihood}
p_{\bY | \bX}(\by | \bx) \propto \exp\left[ -\dfrac{1}{2} \big(\mM(\bx) - \by \big)^\top \bSigma_{\text{obs}}^{-1} \big(\mM(\bx) - \by \big) \right],
\end{equation}
with its score function being
\begin{equation}\label{eq.obs.likelihood.score}
    \bS_{\bY | \bX}(\by | \bx) = \nabla_{\bx} \log p_{\bY | \bX}(\by | \bx) 
    = -\nabla_{\bx} \mM(\bx)^\top \bSigma_{\text{obs}}^{-1} \big( \mM(\bx) - \by \big),
\end{equation}
where $\nabla_{\bx}\mM(\bx)$ denotes the Jacobian of the observation operator. 
 {Note that, unlike the EnKF and ETKF, which do not use gradient information, evaluating this likelihood score requires this Jacobian.}
By Bayes’ rule, the target posterior distribution is
\begin{equation}\label{eq.post.bayes}
    p_{\bX|\bY}(\bx|\by) \propto p_{\bX}(\bx) \, p_{\bY|\bX}(\by|\bx).
\end{equation}
Our objective is to construct the time-dependent posterior score needed for reverse-time SDE sampling, which will be discussed in the following subsections.

\subsection{The exact posterior score with a Gaussian mixture prior}\label{sec.exact.post.score}
We begin by deriving an exact posterior score expression under the assumption that the prior distribution $p_{\bX}(\bx)$ in Eq.~\eqref{eq.post.bayes} is Gaussian mixture model. Specifically, we have the following assumption. 
\begin{assumption}[The Gaussian mixture prior distribution]\label{assump.gmm}
    The prior distribution $p_{\bX}(\bx)$ of the Bayesian problem in Eq.~\eqref{eq.post.bayes} is modeled as a Gaussian mixture model, i.e.,
\begin{equation}\label{eq.gmm.prior.def}
    p_{\bX}(\bx) := \sum_{k=1}^K p_\xi(k) \, \phi(\bx; \bmu_k, \bSigma),
\end{equation}
where $p_{\bX|\xi}(\bx|k) := \phi(\bx; \bmu_k, \bSigma)$ is the PDF of the $k$-th Gaussian component, and $p_\xi(k) = 1/K$ is the (uniform) component weight associated with the latent variable $\xi \in \{1,\dots,K\}$.
\end{assumption}
%
This assumption is reasonable in data assimilation, as it can be interpreted as an ``inflated'' or ``perturbed'' form of the empirical distribution used in the EnKF. Each prior sample $\bx_k$ serves as the mean of a Gaussian component, while a covariance $\bSigma$ introduces local perturbations around each sample. Compared to the raw empirical distribution, which places all mass directly at each prior sample, the added covariance $\bSigma$ expands the support of the prior sample to the surroundings, helping to mitigate the particle degeneracy. The covariance $\bSigma$, typically estimated from the prior ensemble $\{\bx_k\}_{k=1}^K$, also encodes correlations between observed and unobserved variables, thereby enabling information from observations to propagate to unobserved states. 
However, introducing $\bSigma$ directly to each prior sample would artificially inflate the overall variance of the distribution.  
To address this, the Gaussian component means $\bmu_k$ are obtained by shrinking the original prior samples $\bx_k$ toward their ensemble mean $\Bar{\bx}$, with the amount of shrinkage determined by the added covariance $\bSigma$.  
Meanwhile, $\bSigma$ itself is chosen proportional to the sample covariance $\Bar{\bSigma}$ after applying standard localization and inflation.  
The construction is given by
\begin{equation}\label{eq.prior.gm.construct}
    \bmu_k = \sqrt{1-\gamma^2}\,(\bx_k - \Bar{\bx}) + \Bar{\bx}, 
    \qquad 
    \bSigma = \gamma\,\Bar{\bSigma},
\end{equation}
where $\gamma$ is the variance-splitting parameter that controls the Gaussianity of the constructed GM prior.  
Further details of this adjustment are provided in Appendix~\ref{append.gm.prior}.

\vspace{0.1cm}
\begin{remark}
    Note that the following derivation of the posterior score readily extends to a general GMM of the form $\sum_{k=1}^K \pi_k \, \phi(\bx; \bmu_k, \bSigma_k)$, where $\pi_k$ and $\bSigma_k$ vary across all components. For practical data assimilation, however, we restrict to uniform weights and a shared common covariance $\bSigma$. This choice is motivated by two considerations:  
1) prior samples $\{\bx_k\}_{k=1}^K$ are given unweighted, so uniform mixture weights are consistent with ensemble-based methods;  
2) estimating a separate covariance for each component is infeasible with limited ensemble sizes, so we instead use a shared covariance $\bSigma$, estimated from the prior sample covariance $\hat{\bSigma}$.  
\end{remark}
\vspace{0.1cm}

We now turn to deriving the posterior score for the diffusion-based sampling.  To proceed, we first define two processes: the prior forward SDE $\bZ_t$ and the posterior forward SDE $\bZ_t^{\bY}$, both evolving under the same forward diffusion dynamics in Eq.~\eqref{eq.forward.sde} but with different initial conditions: $\bZ_0 \deq \bX$, which is given in Eq.~\eqref{eq.gmm.prior.def} and $\bZ_0^{\bY} \deq \bX | \bY $, defined through the Bayes rule from Eq.~\eqref{eq.post.bayes}. The posterior score of our target, which drives the reverse SDE to generate the  {posterior} sample, is then given by 
\begin{equation}\label{eq.post.diff.score.def}
    \bS_{\bZ_t^{\bY}} (\bz_t, t, \by) = \nabla_z \log p_{\bZ_t^{\bY}}(\bz_t, \by)
\end{equation}
where $p_{\bZ_t^{\bY}}(\bz_t, \by)$ is the probability density of $\bZ_t^{\bY}$ under $\bY=\by$.

Since both $\bZ_t$ and $\bZ_t^{\bY}$ follow the same forward diffusion dynamics, they will share the same transition kernel:
\begin{equation}
    p_{\bZ_t | \bZ_0} (\bz_t | \bz_0) \equiv p_{\bZ_t^{\bY} | \bZ_0^{\bY}}(\bz_t | \bz_0).
\end{equation}
As a result, the density of the posterior process $p_{\bZ_t^{\bY}}(\bz_t, \by)$ can be expressed by conditioning the observation $\bY$ directly on the prior diffusion process $\bZ_t$:
\begin{equation}\label{eq:zzt}
\begin{aligned}
      p_{\bZ_t^{\bY}}(\bz_t, \by) 
      & = \int q_{\bZ_t^{\bY} | \bZ^{\bY}_0}(\bz_t | \bz_0)  p_{\bZ^{\bY}_0} (\bz_0, \by) d \bz_0 
      = \int p_{\bZ_{t} | \bZ_{0} }(\bz_t | \bz_0)  p_{\bX | \bY}(\bz_0 | \by) d\bz_0\\
      & = \int p_{\bZ_{t} | \bZ_{0}, \bY}(\bz_t | \bz_0)  p_{\bZ_0 | \bY} (\bz_0 | \by) d\bz_0 = p_{\bZ_t | \bY}(\bz_t, \by).\\
\end{aligned}
\end{equation}
Thus, the target posterior score for $\bZ^{\bY}_t$ can be written as
\begin{equation}\label{eq:target-score-def}
    S_{\bZ^{\bY}_t}(\bz_t | \by) = S_{\bZ_t|\bY}(\bz_t | \by)  := \nabla_{\bz_t} \log p_{\bZ_t|\bY}(\bz_t) = \frac{\nabla_{\bz_t} p_{\bZ_t|\bY}(\bz_t|\by)}{p_{\bZ_t|\bY}(\bz_t |\by)}.
\end{equation}
Compared with directly working with $\bZ^{\bY}_t$, which is difficult since the initial posterior distribution $\bZ_0^{\bY} \deq \bX | \bY $ is the unknown posterior, the representation in terms of $p_{\bZ_t | \bY}(\bz_t | \by)$ is advantageous. It enables the use of Bayes’ rule together with the analytical properties of the prior process $\bZ_t$, which is explicitly tractable under Assumption~\ref{assump.gmm}. This tractability is critical to deriving the exact posterior score expression.

The following proposition characterizes the transition densities of the prior diffusion process $\bZ_t$ under Assumption \ref{assump.gmm}, which will be used to derive the closed form of the posterior score.

\begin{proposition}[The forward and reverse transitional probabilities]\label{prop.prior.trans.density}
Let $\bZ_t$ follow the forward SDE from Eq.~\eqref{eq.forward.sde}. If the prior distribution is a GM distribution
\begin{equation}
    p_{\bZ_0}(\bz_0) = \sum_k  \pi_k \, \phi(\bz_0 ; \bmu_k, \bSigma),
\end{equation}
then we have the following properties:
\begin{enumerate}
\item The marginal density at time $t$ is a Gaussian mixture defined by
\begin{equation}\label{eq.gmm.forward.marginal}
    p_{\bZ_t} (\bz_t) = \sum_k p_{\xi}(k)\, \phi(\bz_t; \bmu_{t,k}, \bSigma_{t}),
\end{equation}
where each component evolves as
\begin{align}
    \bmu_{t,k} &= \alpha_t \bmu_k, \label{eq.gmm.forward.mean}\\
    \bSigma_{t} &= \alpha_t^2 \bSigma + \beta_t^2 \bI. \label{eq.gmm.forward.cov}
\end{align}
\item The reverse transition distribution is also a Gaussian mixture:
\begin{equation}\label{eq.gmm.reverse.kernel}
    p_{\bZ_0 | \bZ_t}(\bz_0 | \bz_t) = \sum_k p_{\xi | \bZ_t}(k | \bz_t) \, p_{\bZ_0 | \bZ_t, \xi}(\bz_0 | \bz_t, k),
\end{equation}
where $p_{\bZ_0 | \bZ_t, \xi}(\bz_0 | \bz_t, k) \equiv \phi(\bz_0; \bmu_{0|t,k}(\bz_t), \bSigma_{0|t,k})$ is a Gaussian with 
\begin{align}
    \bmu_{0|t,k}(\bz_t) &= \bmu_k + \alpha_t \bSigma \left( \alpha_t^2 \bSigma + \beta_t^2 \bI \right)^{-1} (\bz_t - \alpha_t \bmu_k), \label{eq.gmm.reverse.mean}\\
    \bSigma_{0|t} &= \bSigma - \alpha_t^2 \bSigma \left( \alpha_t^2 \bSigma + \beta_t^2 \bI \right)^{-1} \bSigma.\label{eq.gmm.reverse.cov}
\end{align}
and the mixture weights are
\begin{equation}\label{eq.gmm.reverse.weight}
    p_{\xi | \bZ_t}(k | \bz_t) = \frac{\pi_k \, \phi(\bz_t; \bmu_{t,k}, \bSigma_{t})}{\sum_j \pi_j \, \phi(\bz_t; \bmu_{t,j}, \bSigma_{t})}.
\end{equation}
\end{enumerate}
 {The proof is given in~\ref{app.prop.derivation}.}
\end{proposition}
\begin{proposition}[The prior score function ]\label{prop.gmm.prior.score}
Under Assumption~\ref{assump.gmm}, the score of the prior distribution at time $t$ is
\begin{equation}\label{eq.gmm.prior.score}
    \bS_{\bZ_t}(\bz_t) = \nabla_{\bz_t} \log p_{\bZ_t}(\bz_t) 
    = \sum_k p_{\xi | \bZ_t}(k | \bz_t) \, \bS_{\bZ_t | k}(\bz_t),
\end{equation}
where $\bmu_{t,k}$ and $\bSigma_{t}$ are given in Eq.\eqref{eq.gmm.forward.mean} and Eq.\eqref{eq.gmm.forward.cov}, $p_{\xi | \bZ_t}(k | \bz_t)$ is the prior mixture weight given by Eq.\eqref{eq.gmm.reverse.weight} and $\bS_{\bZ_t | k}(\bz_t)$ is the k-th Gaussian component score given by
\begin{equation}\label{eq.gmm.prior.score.component}
    \bS_{\bZ_t | k}(\bz_t) = -\bSigma_{t}^{-1} (\bz_t - \bmu_{t,k}).
\end{equation}
 {The proof is given in~\ref{app.prop.derivation}.}
\end{proposition}
%


This prior score function generalizes the EnSF prior score in Eq.~\eqref{eq.ensf.prior.derive}. Specifically, EnSF prior score in Eq.~\eqref{eq.ensf.prior.derive} can be recovered as the degenerate case of $\bSigma=\bzero$ where the mixture collapses to delta functions centered at ensemble members. 
In addition, the presence of $\bSigma$ introduces anisotropic correlations within the prior component score terms $\bS_{\bZ_t | k}(\bz_t)$, which helps the update of unobserved variables through the covariance structure in the reverse SDE. Based on the above two propositions, we have the following theorem on the posterior score function.

\begin{theorem}[The exact posterior score]\label{thm.post.score.expression}
Let $\bZ_t$ evolve under the forward SDE in Eq.~\eqref{eq.forward.sde} with initial distribution $p_{\bZ_0}(\bz_0) = \sum_k \pi_k \,\phi(\bz_0;\bmu_k,\bSigma)$, and let $\bY$ be defined by Eq.~\eqref{eq.bayes.static.obs}. Then the posterior score is
\begin{equation}\label{eq.post.score.expression.final}
    \bS_{\bZ_t | \bY}(\bz_t | \by) 
    = \sum_k p_{\xi | \bZ_t, \bY}(k | \bz_t, \by) \, \bS_{\bZ_t | k}(\bz_t) 
    + \bJ(t) \, \bbE[\bS_{\bY|\bX}(\by | \bZ_0) | \bZ_t = \bz_t, \bY = \by],
\end{equation}
\vspace{-0.2cm}
where
\vspace{-0.0cm}
\begin{itemize}
    \item $p_{\xi|\bZ_t,\bY}(k|\bz_t,\by)$ are posterior mixture weights;
    \item $\bS_{\bZ_t|k}(\bz_t)$ is prior component score given in Eq.~\eqref{eq.gmm.prior.score.component}; 
    \item $\bJ(t)$ is the Jacobian of the mean of reverse-time transition kernel in Eq.~\eqref{eq.gmm.reverse.mean};
    \item $\bS_{\bY|\bX}(\by | \cdot)$ is the likelihood score given in Eq.~\eqref{eq.obs.likelihood.score};
    \item $\bbE[\bS_{\bY|\bX}(\by | \bZ_0) | \bZ_t = \bz_t, \bY = \by]$ is the expected likelihood score at time zero given the observation and diffusion state. 
\end{itemize}
Specifically, the expression of $\bJ(t)$ is given as
\begin{equation}\label{eq.jt.expression}
    \bJ(t) = \alpha_t \bSigma \left( \alpha_t^2 \bSigma + \beta_t^2 \bI_d \right)^{-1}.
\end{equation}
 {The proof is given in~\ref{app.post.score.derivation}.}
\end{theorem}

Similar to the standard practice of conditional sampling in diffusion models, where the learned prior score is combined with a time-scaled likelihood score term in the reverse SDE, our posterior score likewise decomposes into two terms, i.e., the prior score $\sum_k p_{\xi | \bZ_t,\bY}(k | \bz_t,\by)\,\bS_{\bZ_t | k}(\bz_t)$ and the likelihood guidance term $\bJ(t)\,\bbE[\bS_{\bY | \bX}(\by | \bZ_0) | \bZ_t=\bz_t,\bY=\by]$.  {We note that $\alpha_t$ and $\beta_t$ here are the fixed diffusion noise schedule from
Section~\ref{sec.background}, not tunable parameters of the method.} Compared with the standard formulation, our derivation yields several key insights:
\begin{itemize}[leftmargin=20pt]
    \item {\em Observation-dependent prior score.} Under the GM prior, the observation $\bY$ enters not only through the likelihood term but also through the mixture weights $p_{\xi | \bZ_t,\bY}(k | \bz_t,\by)$ of the prior score. In contrast, in standard conditional diffusion models, the dependence on $\bY$ only appears in the likelihood guidance term. This additional dependence provides a new channel by which observations $\bY$ influence posterior sampling. Note that this $\bY$-dependency is gone when we only use a single Gaussian prior. 
    
    \item {\em Explicit likelihood scaling.} The likelihood score is scaled by a time-dependent factor $\bJ(t)$ in Eq.~\eqref{eq.jt.expression}, which depends explicitly on the prior covariance $\bSigma$ and the diffusion schedule $\alpha_t$ and $\beta_t$. This eliminates the need for an ad hoc tunable guidance strength and ensures that the scaling adapts automatically to the ensemble covariance structure.
    
    \item {\em Exact likelihood score expression.} Instead of relying on a single-point evaluation of $\bS_{\bY | \bX}(\by | \cdot)$ during posterior sampling, as is common in conditional diffusion models, our result shows that the correct likelihood score is given by the expectation of $\bS_{\bY | \bX}(\by | \cdot)$ with respect to the distribution $p_{\bZ_0 | \bZ_t,\bY}(\bz_0 | \bz_t,\by)$. Although this distribution is generally intractable, the exact formulation provides theoretical guidance for constructing practical approximations with reduced error, as developed in Section~\ref{sec.approx.cond.exp}, and for assessing their accuracy.
\end{itemize}

 {In high dimensions, the covariance $\alpha_t^2 \bSigma + \beta_t^2 \bI_d$ arising in the
reverse-time SDE is not inverted explicitly at each step. Since it shares the eigenvectors of
the localized prior covariance $\bSigma$ for all $t$, a single eigendecomposition of $\bSigma$ per
assimilation cycle suffices, and each reverse-SDE step then applies the inverse in $O(d^2)$ by
shifting eigenvalues, giving a dominant per-cycle cost of one $O(d^3)$ eigendecomposition. For
dimensions at which a dense covariance cannot be formed, this global treatment is replaced by
domain localization, as in the LETKF, decomposing the analysis into independent local
subproblems of dimension $d_{\mathrm{local}} \ll d$ and reducing the cost to
$O(d_{\mathrm{local}}^3)$ per patch. A complementary strategy for scaling the sampling step is
patch-wise (tiled) denoising, as used in recent diffusion-model work \cite{multidiffusion},
where a large field is decomposed into overlapping patches denoised in parallel and fused in
their overlaps; adapting this so that the decomposition preserves the covariance-mediated update
of unobserved variables is left to future work.}

%% file: 4_score_approximation.tex
\subsection{Approximation of the posterior score function}\label{sec.approx.post.score}
In this section, we develop approximations for the exact posterior score in Eq.~\eqref{eq.post.score.expression.final}. For clarity, we first rewrite it in the simplified form
\begin{equation}\label{eq.post.score.simple}
    \bS_{\bZ_t|\bY} (\bz_t, \by)
    = \sum_k w^{post}(k,\bz_t,\by)\, \bS_{\bZ_t | k}(\bz_t) 
    + \bJ(t)\,\bmu^{\bS_{\bY | \bX}}(\bz_t, \by),
\end{equation}
where $\bJ(t)$ and $\bS_{\bZ_t | k}(\bz_t)$ admit closed-form expressions (Eq.~\eqref{eq.jt.expression} and Eq.~\eqref{eq.gmm.prior.score.component}). The remaining challenge is to approximate the following two intractable terms:
\begin{align}
    \text{Posterior mixture weights: }& w^{post}(k,\bz_t,\by)  \equiv p_{\xi|\bZ_t,\bY}(k|\bz_t,\by), \label{eq.w.post.def} \\
    \text{Conditional expectation: }& \bmu^{\bS_{\bY | \bX}} (\bz_t,\by)
     \equiv \mathbb{E}\!\left[\bS_{\bY|\bX}(\by | \bZ_0)\,\big|\, \bZ_t=\bz_t,\;\bY=\by\right]\label{eq.s.mus.def}.
\end{align}
Our strategy is as follows.  
\begin{itemize}[leftmargin=15pt]
    \item {\em Approximating the posterior mixture weights.} In Section~\ref{sec.approx.post.weight}, we approximate the posterior weights $w^{post}(k,\bz_t,\by)$, with the key step being the introduction of the $\bY$-dependence via a tractable linearization of the observation operator.  
    \item {\em Approximating the conditional expectation with a single-point evaluation.} In Section~\ref{sec.approx.cond.exp}, we approximate the conditional expectation $\bmu^{\bS_{\bY | \bX}} (\bz_t,\by)$ by replacing the full expectation with a single evaluation $\bS_{\bY|\bX}(\by | \Bar{\bmu}_0^*(\bz_t))$, where the evaluation point $\Bar{\bmu}_0^*(\bz_t)$ is obtained through a Kalman-type update.  
    \item {\em Iteratively refining the evaluation point of the likelihood score.} Finally, since the true $\Bar{\bmu}_0^*(\bz_t)$ depends on the unknown posterior $p_{\bX|\bY}$, we propose in Section~\ref{sec.iterative.refinement} an iterative refinement scheme that repeatedly solves the reverse SDE to progressively improve the posterior information used in $\Bar{\bmu}_0^*(\bz_t)$. 
\end{itemize}
 {
Throughout this section, both the observation dependence introduced into the prior-score weights and the evaluation point of the likelihood score are exact quantities, following from Theorem~\ref{thm.post.score.expression}; the local linearization and the Kalman-type update are used only as tractable approximations of these exact quantities, and do not constrain the posterior, which is generated by the reverse SDE and can remain non-Gaussian, as discussed in Remarks~\ref{remark.m.linearization} and~\ref{remark.gaussian.aprox}.
}

\subsubsection{Approximation of the posterior mixture weight $w^{post}(k, \bz_t, \by)$}\label{sec.approx.post.weight}

We begin by approximating the posterior mixture weight $w^{post}(k,\bz_t,\by)$. The main challenge arises from the nonlinear observation operator $\mM(\cdot)$, which renders the Gaussian integral intractable. To overcome this, we apply a local linearization of $\mM(\cdot)$, which makes the integral analytically solvable. It is important to note that this linearization only modifies how $\bY$ enters the weighting function, and does not impose restrictions on the form of the final posterior distribution (see Remark~\ref{remark.m.linearization}).

Applying Bayes’ rule, we first decompose $w^{post}(k,\bz_t,\by) = p_{\xi | \bZ_t,\bY}(k | \bz_t,\by)$ as
\begin{equation}\label{eq.w.post.bayes.decom.p}
    w^{post}(k,\bz_t,\by) \propto p_{\xi | \bZ_t}(k | \bz_t)\, p_{\bY | \bZ_t, \xi}(\by | \bz_t , k),
\end{equation}
where the proportionality follows from the marginalization $\sum_k p_{\xi | \bZ_t,\bY}(k | \bz_t,\by) = 1$. For clarity, we rewrite this decomposition in Eq.~\eqref{eq.w.post.bayes.decom.p} as
\begin{equation}\label{eq.w.post.bayes.decom.rewrite}
    w^{post}(k,\bz_t,\by) \propto w^{prior}(k,\bz_t)\, w^{obs}(k,\by,\bz_t),
\end{equation}
where $w^{prior}(k,\bz_t) := p_{\xi | \bZ_t}(k | \bz_t)$ is the prior weighting term appearing in the prior score from  Eq.~\eqref{eq.gmm.prior.score}, with explicit form given in Eq.~\eqref{eq.gmm.reverse.weight}, and and 
\begin{equation}\label{eq.w.obs.def}
    w^{obs}(k,\by,\bz_t) := p_{\bY | \bZ_t, \xi}(\by | \bz_t , k)
\end{equation}
is the observation-dependent correction factor.

To compute $w^{obs}(k,\by,\bz_t) := p_{\bY | \bZ_t, \xi}(\by | \bz_t , k)$, we introduce $\bZ_0$ as an intermediate variable and write it as an integral with respect to $\bZ_0$:
\begin{align}
w^{obs}(k, \by, \bz_t) 
&= \int p_{\bY | \bZ_0, \bZ_t, \xi}(\by | \bz_0, \bz_t, k) \; p_{\bZ_0 | \bZ_t, \xi}(\bz_0 | \bz_t, k) \; d \bz_0\\
&= \int p_{\bY | \bZ_0}(\by | \bz_0) \; p_{\bZ_0 | \bZ_t, \xi}(\bz_0 | \bz_t, k) \; d\bz_0, \label{eq.w.obs.integral}
\end{align}
where $p_{\bY | \bZ_0}(\cdot)$ is the likelihood function defined in Eq.~\eqref{eq.obs.likelihood}, and $p_{\bZ_0 | \bZ_t, \xi}(\bz_0 | \bz_t) = \phi(\bz_0; \bmu_{0|t,k}(\bz_t), \bSigma_{0|t}) $ is the Gaussian reverse transition kernel with $\bmu_{0|t,k}(\bz_t)$ given by Eq.~\eqref{eq.gmm.reverse.mean} and $\bSigma_{0|t}$ given by Eq.~\eqref{eq.gmm.reverse.cov}. 
Given $p_{\bZ_0 | \bZ_t, \xi}(\bz_0 | \bz_t, k)$ is now a tractable Gaussian, the integral in Eq.~\eqref{eq.w.obs.integral} still remains intractable, due to the nonlinear $\mM(\cdot)$ in $p_{\bY | \bZ_0}(\by | \bz_0)$. Therefore, we consider linearize $\mM(\cdot)$ around a chosen expansion point $\bmu_k^*$:
\begin{equation}\label{eq.m.linearization}
    \mM(\bx) \approx \mM(\bmu_k^*) + \bJ_{\mM}(\bmu_k^*)(\bx - \bmu_k^*) = \ba_k + \bH_k^* \bx,
\end{equation}
where $\bH_k^* = \bJ_{\mM}(\bmu_k^*)$ is the Jacobian and $\ba_k = \mM(\bmu_k^*) - \bH_k^* \bmu_k^*$.

\vspace{0.1cm}
\begin{remark}[Comparison with EKF]\label{remark.m.linearization}
Although the linearization in Eq.~\eqref{eq.m.linearization} resembles that in the Extended Kalman Filter (EKF), the purposes are fundamentally different. 
EKF applies linearization globally to propagate Gaussian statistics, thereby constraining the posterior to remain Gaussian, which can introduce significant errors under highly nonlinear $\mM(\cdot)$. 
In contrast, our linearization is applied only locally within the Gaussian integral of $w^{obs}(k,\by,\bz_t)$ to make it tractable and to introduce observation-dependence. The posterior samples are still generated by solving the reverse SDE, thereby preserving the flexibility to capture non-Gaussian features that EKF cannot.  
\end{remark}
\vspace{0.1cm}

Under this linearization, the integral in Eq.~\eqref{eq.w.obs.integral} reduces to a tractable  Gaussian form 
and can be evaluated explicitly, yielding
\begin{equation}\label{eq.w.obs.expression}
    w^{obs}(k,\by,\bz_t) 
    = (\by_k^* - \bH_k^* \bmu_{0|t,k}(\bz_t))^\top \big(\bH_k^* \bSigma_{0|t} \bH_k^{*\top} + \bSigma_{obs}\big)^{-1} 
    (\by_k^* - \bH_k^* \bmu_{0|t,k}(\bz_t)),
\end{equation}
where $\by_k^* = \by - \ba_k$ and $\bmu_{0|t,k}(\bz_t)$ is the reverse conditional mean from Eq.~\eqref{eq.gmm.reverse.mean}.
A natural choice for the linearization point of $\mM(\cdot)$ is $\bmu_k^* = \bmu_{0|t,k}(\bz_t)$, i.e., the mean of each Gaussian reverse transitional kernel $p_{\bZ_0 | \bZ_t, \xi}(\bz_0 | \bz_t, k)$. Substituting this choice into Eq.~\eqref{eq.w.obs.expression} simplifies the observation factor to
\begin{equation}\label{eq.w.obs.expression.concrete}
    w^{obs}(k,\by,\bz_t) 
    = \big(\by - \mM(\bmu_{0|t,k}(\bz_t))\big)^\top 
    \big(\bH_k \bSigma_{0|t} \bH_k^\top + \bSigma_{obs}\big)^{-1} 
    \big(\by - \mM(\bmu_{0|t,k}(\bz_t))\big),
\end{equation}
with $\bH_k = \bJ_{\mM}(\bmu_{0|t,k}(\bz_t))$.

In summary, the posterior weight $w^{post}(k,\bz_t,\by)$ is obtained as the product of the prior weight $w^{prior}(k,\bz_t)$ and an observation-dependent correction factor $w^{obs}(k,\by,\bz_t)$. The correction factor becomes analytically tractable by linearizing the observation operator $\mM(\cdot)$, with the linearization point chosen based on the location of the reverse Gaussian kernel $p_{\bZ_0 | \bZ_t, \xi}(\bz_0 | \bz_t, k)$. Alternatively, simpler choices of linearization of $\mM(\cdot)$ are discussed in \ref{sec.w.obs.extra}.

\subsubsection{Approximation of the conditional expectation $\bmu^{\bS_{\bY | \bX}} (\bz_t,\by)$}\label{sec.approx.cond.exp}
We next approximate the conditional expectation $\bmu^{\bS_{\bY | \bX}} (\bz_t,\by)$ defined in Eq.~\eqref{eq.s.mus.def}. 
By definition, this requires integrating likelihood score function $\bS_{\bY|\bX}(\by | \cdot)$ against the conditional density $p_{\bZ_0 | \bZ_t,\bY}(\bz_0 | \bz_t,\by)$, which is generally intractable due to the nonlinear observation operator $\mM(\cdot)$. 
A direct Monte Carlo approach would also be infeasible, since it would require first sampling from $p_{\bZ_0 | \bZ_t,\bY}$ and then averaging $\bS_{\bY|\bX}(\by|\cdot)$ on those samples at every step on each path of the reverse SDE, which lead to prohibitive computational cost.

To make the approximation computationally feasible, we replace the conditional expectation with a single evaluation of the likelihood score function  $\bS_{\bY|\bX}(\by | \cdot)$ at an appropriately chosen point. Specifically, we first move the expectation inside the likelihood score and approximate it as 
\begin{equation}\label{eq.cond.exp.approx.1}
    \bmu^{\bS_{\bY | \bX}} (\bz_t, \by) 
    \approx \bS_{\bY|\bX}\!\left(\by \,\big|\, \mathbb{E}\!\left[\bZ_0 \,\big|\, \bZ_t=\bz_t,\;\bY=\by\right]\right).
\end{equation}
The approximation error arises from the nonlinearity of $\bS_{\bY|\bX}(\by|\cdot)$ as a result of the nonlinear observation model $\mM(\cdot)$. This error is often referred to as Jensen’s gap. If $\mM(\cdot)$ is linear, then $\bS_{\bY|\bX}(\by|\cdot)$ is also linear, and the approximation becomes exact.

The next task is to determine the evaluation point in approximation from Eq.~\eqref{eq.cond.exp.approx.1}, namely 
\begin{equation}\label{eq.mu.bar.def}
    \Bar{\bmu}_0(\bz_t,\by) := \bbE[\bZ_0 | \bZ_t=\bz_t,\, \bY=\by].
\end{equation}
which is the conditional mean of the initial variable $\bZ_0$ given a observation $\bY$ and the noisy state $\bZ_t$. 
Using the conditional independence of $\bY$ and $\bZ_t$ given $\bZ_0$, we can formulate the joint Bayesian model for $\bbE[\bZ_0 | \bZ_t,\, \bY]$ as 
\begin{equation}\label{eq.bayes.joint}
\begin{aligned}
    \text{Prior:}& \;\; \bZ_0 \deq \bX,\\
    \text{Observation:}& \;\; \bY = \mM(\bZ_0) + \bveps_1, \quad \bveps_1 \sim \mN(0, \bSigma_{\text{obs}}),\\
    \text{Diffusion:}& \;\; \bZ_t = \alpha_t \bZ_0 + \beta_t \bveps_2, \quad \bveps_2 \sim \mN(0, \bI_d).
\end{aligned}
\end{equation}
where $\bveps_1 \perp \bveps_2$. 
This formulation can be interpreted as assimilating two observations into the prior $\bZ_0$: the real-world observation $\bY$ and the diffusion observation $\bZ_t$.
Assimilating $\bY$ first yields the posterior $\Bar{\bZ}_0 := \bZ_0 | \bY$, which then serves as the prior for assimilating $\bZ_t$:
\begin{equation}\label{eq.bayes.diffusion}
\begin{aligned}
    \text{Prior:}& \;\; \Bar{\bZ}_0 \deq \bZ_0 | \bY,\\
    \text{Observation:}& \;\; \bZ_t = \alpha_t \Bar{\bZ}_0 + \beta_t \bveps_2.
\end{aligned}
\end{equation}
Here, the prior $\Bar{\bZ}_0$ is exactly the unknown target posterior $p_{\bZ_0|\bY} = p_{\bX|\bY}$, while the diffusion observation follows a linear Gaussian model.

In light of the linear observation model in Eq.~\eqref{eq.bayes.diffusion}, we approximate the unknown prior $\Bar{\bZ}_0 \deq \bX | \bY$ with a Gaussian distribution $\Bar{\bZ}_0^* \sim \mN(\bmu^*, \bSigma^*)$, which we refer to as the \emph{reference posterior}. 
With this approximation, the posterior mean in Eq.~\eqref{eq.bayes.diffusion} can be computed analytically via the Kalman update formula:
\begin{equation}\label{eq.mubar.expression}
    \Bar{\bmu}_0(\bz_t, \by) \approx \Bar{\bmu}_0^*(\bz_t) = \bmu^* + \alpha_t \bSigma^*(\alpha_t^2 \bSigma^* + \beta_t^2 \bI_d)^{-1}(\bz_t - \alpha_t \bmu^*).
\end{equation}
Substituting this expression into Eq.~\eqref{eq.cond.exp.approx.1} yields the final approximation
\begin{equation}\label{eq.cond.exp.approx.final}
    \bmu^{\bS_{\bY | \bX}} (\bz_t,\by) 
    \approx \bS_{\bY | \bX}(\by | \Bar{\bmu}_0^*(\bz_t)),
\end{equation}
where the evaluation point $\Bar{\bmu}_0^*(\bz_t)$ depends on the choice of the reference posterior $\mN(\bmu^*, \bSigma^*)$.

From the expression of our approximated evaluation point $\Bar{\bmu}_0^*(\bz_t)$ in Eq.~\eqref{eq.mubar.expression}, we see that it interpolates between $\bz_t$ at $t=0$ and the reference mean $\bmu^*$ at $t=T$. The chosen reference posterior $\mN(\bmu^*, \bSigma^*)$ determines the endpoint at $t=T$ with $\bmu^*$, while the interpolation path is controlled by $\bSigma^*$ together with the noise schedule $(\alpha_t, \beta_t)$ (see Remark~\ref{remark.mubar} for a more detailed discussion of this interpolation behavior). In the next section, we introduce an iterative procedure to progressively improve the choice of the reference posterior $\mN(\bmu^*, \bSigma^*)$.

\vspace{0.1cm}
\begin{remark}[Comparison with Laplace approximation]\label{remark.gaussian.aprox}
Approximating the true posterior $\Bar{\bZ}_0 \deq \bX | \bY$ by a Gaussian  $\mN(\bmu^*, \bSigma^*)$ is reminiscent of the classical Laplace approximation, where a non-Gaussian posterior is approximated directly as a Gaussian via a second-order Taylor expansion of the log-posterior. 
The crucial difference is that here the Gaussian assumption is used only as an auxiliary tool to compute the evaluation point of the likelihood score in Eq.~\eqref{eq.cond.exp.approx.1}.
And the posterior samples themselves are still generated by the reverse SDE, so the final distribution remains non-Gaussian.
\end{remark}
\vspace{0.1cm}
\begin{remark}[Interpolation behavior of $\Bar{\bmu}_0^*(\bz_t)$]\label{remark.mubar}
From Eq.~\eqref{eq.mubar.expression}, the interpolating behavior of $\Bar{\bmu}_0^*(\bz_t)$ can be summarized as follows:
\begin{itemize}[leftmargin=15pt]
    \item At $t=0$: We have $\Bar{\bmu}_0^*(\bz_t) = \bz_t$. From the Bayesian formulation in Eq.~\eqref{eq.bayes.diffusion}, it also follows that $\Bar{\bmu}_0(\bz_t) = \bz_t$, which is consistent with $\Bar{\bmu}_0^*(\bz_t)$. The original EnSF, which always evaluates at $\bz_t$, is also consistent in this regime. 
    
    \item At $t=T$: We obtain $\Bar{\bmu}_0^*(\bz_t) = \bmu^*$, while the true value is $\Bar{\bmu}_0 (\bz_t,\by) = \mathbb{E}[\bX | \bY]$. In both cases, the evaluation point becomes constant and independent of $\bz_t$, though the constants differ. Consequently, both the true $\Bar{\bmu}_0 (\bz_t,\by)$ and approximate $\Bar{\bmu}_0^*(\bz_t)$ yield an almost constant posterior score when $t$ is close to $T$, which does not distort the intermediate distribution from Gaussian during the reverse SDE. In contrast, the original EnSF continues to evaluate $\bS_{\bY|\bX}(\by | \cdot)$ at $\bz_t$, which can introduce spurious distortions at the beginning of the reverse SDE, precisely when the distribution should remain close to Gaussian. 
    
    \item For $0 < t < T$: At intermediate times, $\Bar{\bmu}_0^*(\bz_t)$ interpolates between $\bz_t$ and $\bmu^*$ in an \emph{anisotropic} manner determined by $\bSigma^*$. In the special case of an isotropic covariance, $\bSigma^* = \sigma^2 \bI$, the interpolation simplifies to a scalar-weighted linear combination:
    \begin{equation}\label{eq.mubar.iso}
        \Bar{\bmu}_0^*(\bz_t) = (1-\alpha_t \lambda_t)\,\bmu^* + \lambda_t\,\bz_t, \quad \text{with} \quad \lambda_t = \frac{\alpha_t \sigma^2}{\alpha_t^2 \sigma^2 + \beta_t^2}. 
    \end{equation}
\end{itemize}
\end{remark}

\subsubsection{The iterative refinement of the likelihood score evaluation point $\Bar{\bmu}_0^*(\bz_t)$ }\label{sec.iterative.refinement}

After the approximations in Sections~\ref{sec.approx.post.weight} and \ref{sec.approx.cond.exp}, the posterior score takes the computable form
\begin{equation}\label{eq.posterior.score.approx.final}
    \Hat{\bS}_{\bZ_t|\bY}(\bz_t|\by) 
    = \sum_k w^{prior}(k,\bz_t)\, w^{obs}(k,\bz_t,\by)\, \bS_{\bZ_t|k}(\bz_t) 
    + \bJ(t)\, \bS_{\bY|\bX}\!\big(\by \,\big|\, \Bar{\bmu}_0^*(\bz_t)\big), 
\end{equation}
where $w^{prior}(k,\bz_t)$ is given in Eq.~\eqref{eq.gmm.reverse.weight}, $w^{obs}(k,\bz_t,\by)$ in Eq.~\eqref{eq.w.obs.expression}, $\bS_{\bZ_t|k}(\bz_t)$ in Eq.~\eqref{eq.gmm.prior.score.component}, $\bJ(t)$ in Eq.~\eqref{eq.jt.expression}, and $\Bar{\bmu}_0^*(\bz_t)$ in Eq.~\eqref{eq.mubar.expression}. 
Thus, posterior sampling can already be carried out given a choice of reference posterior $\mN(\bmu^*, \bSigma^*)$. 
The only missing piece is how to determine this reference posterior $\mN(\bmu^*, \bSigma^*)$ in practice.

By construction, $\mN(\bmu^*, \bSigma^*)$ serves as a Gaussian approximation to the unknown posterior $\Bar{\bZ}_0 \deq \bX | \bY$. 
Since this target distribution is unavailable, we introduce an \emph{iterative refinement procedure} to estimate it. 
We begin with an initial guess based on the prior samples $\{\bx_k\}_{k=1}^K \sim p_{\bX}$, fitting a Gaussian $\mN(\bmu^*, \bSigma^*)$ that initially contains no observational information. 
Using this reference, we run the reverse SDE to generate an ensemble of approximate posterior samples $\{\bx_k^*\}_{k=1}^K$. 
From these samples, we update the Gaussian approximation $\mN(\bmu^*, \bSigma^*)$, which now incorporates the influence of $\bY$ through the likelihood score $\bS_{\bY | \bX}(\by | \cdot)$, thereby improving its alignment with the true posterior $\bX | \bY$. 
This process is then repeated, iteratively refining both the Gaussian reference and the ensemble until convergence, i.e., when successive Gaussian fits no longer change significantly. 
For numerical stability, two smoothing parameters $\eta_1$ and $\eta_2$ are introduced when updating $\bmu^*$ and $\bSigma^*$ in Alg.~\ref{alg:iensf}.

\subsection{Summary of the full IEnSF Algorithm}\label{sec.iensf.sum}
We now combine the approximations from the previous sections into a complete algorithm for the data assimilation (DA) update step. The pseudo-code for the Iterative Ensemble Score Filter (IEnSF) is provided in Alg.~\ref{alg:iensf}.  
To highlight the differences, let us compare our new posterior score approximation in Eq.~\eqref{eq.posterior.score.approx.final} with the original EnSF approximation:
\begin{equation}\label{eq.ensf.post.score}
    \text{EnSF:}\quad 
    \Hat{\bS}_{\bZ_t|\bY}(\bz_t | \by) 
    = \sum_k w^{prior}(k,\bz_t)\, \bS_{\bZ_t | k}(\bz_t) 
    + h(t)\, \bS_{\bY|\bX}(\by | \bz_t),
\end{equation}
where $h(t) = 1-t$ is the chosen damping function.  
Our method introduces three main improvements over the original EnSF:  
\begin{itemize}[leftmargin=15pt]
    \item {\em Time-scaling of the likelihood score.} 
    EnSF uses a heuristic damping factor $h(t)=1-t$, chosen only to satisfy $h(0)=1$ and $h(1)=0$.  
    In contrast, IEnSF employs the time-scaling $\bJ(t)$, analytically derived from Theorem~\ref{thm.post.score.expression}.  
    This formulation not only guarantees the correct boundary behavior but also adapts dynamically to the estimated prior covariance $\bSigma$, which can vary across DA cycles.  

    \item {\em Observation-informed prior score weighting.}  
    EnSF approximates mixture weights using only $w^{prior}(k,\bz_t) = p_{\xi|\bZ_t}(k|\bz_t)$, thereby ignoring the observation $\by$.  
    IEnSF introduces an additional factor $w^{obs}(k,\bz_t,\by)$, approximating $p_{\bY|\bZ_t,\xi}(\by|\bz_t,k)$ from the Bayesian decomposition in Eq.~\eqref{eq.w.post.bayes.decom.p}.  
    This reweighting directs the posterior approximation toward ensemble members consistent with the observation, yielding a closer match to the true posterior distribution.  

    \item {\em Refined likelihood score evaluation.}  
    EnSF evaluates the likelihood score $\bS_{\bY|\bX}(\by|\cdot)$ at $\bz_t$.  
    IEnSF instead evaluates at $\Bar{\bmu}_0^*(\bz_t)$, obtained from the Gaussian update in Eq.~\eqref{eq.mubar.expression} and refined iteratively with observational information.  
    This refinement reduces distortions in the early stages of the reverse SDE, when the distribution should remain close to Gaussian, and incorporates explicit observation dependence for better theoretical consistency.  
\end{itemize}

Together, these three improvements yield a posterior score approximation that is theoretically grounded and better aligned with the true Bayesian update.

\begin{algorithm}[t]
\caption{IEnSF for a single DA update step}\label{alg:iensf}
\setstretch{1.2}
\begin{algorithmic}[1]
\Require Prior ensemble $\{ \bx_k \}_{k=1}^K$, observation $\by$, max refinement iteration $M$, posterior reference smoothing $(\eta_1,\eta_2)\in[0,1]^2$
\Ensure Posterior ensemble $\{ \bx_k^{*} \}_{k=1}^K$
\vspace{0.25em}
\State \textbf{Construct GMM prior:} Construct the GM means $\{ \bmu_k \}_{k=1}^K$ and shared covariance $\bSigma$ from $\{ \bx_k \}_{k=1}^K$ by Eq.~\eqref{eq.prior.gm.construct}.
\State \textbf{Initialize reference posterior:} $\bmu^* \gets \text{mean}(\{ \bx_k \})$, \; $\bSigma^* \gets \text{cov}(\{ \bx_k \})$.
\For{$m = 1,\ldots,M$}
    \State \textbf{Posterior sampling:} Using the estimated posterior score in Eq.~\eqref{eq.posterior.score.approx.final} with reference $\mN(\bmu^*,\bSigma^*)$, simulate the reverse SDE/ODE to obtain the posterior samples $\{ \bx_k^{*} \}_{k=1}^K$.
    \State \textbf{Fit a posterior Gaussian from samples:} $\;\Bar{\bmu} \gets \text{mean}(\{ \bx_k^{*} \})$, \; $\Bar{\bSigma} \gets \text{cov}(\{ \bx_k^{*} \})$.
    \State \textbf{Convergence check:} If $d(\mN(\Bar{\bmu}, \Bar{\bSigma}), \mN(\bmu^*, \bSigma^*)) \le \text{tol}$ , \textbf{break}.
    \State \textbf{Reference update (smoothed):} 
    \vspace{-0.3cm}
    \begin{equation*}
        \bmu^* \gets (1-\eta_1)\,\bmu^* + \eta_1\,\Bar{\bmu}, 
        \qquad
        \bSigma^* \gets (1-\eta_2)\,\bSigma^* + \eta_2\,\Bar{\bSigma}. \vspace{-0.3cm}
    \end{equation*}
\EndFor
\State \textbf{Return} $\{ \bx_k^{*} \}_{k=1}^K$.
\end{algorithmic}
\end{algorithm}
 {The tunable parameters of IEnSF are the variance-splitting parameter $\gamma$, and the
localization and inflation applied to $\bar{\bSigma}$; these are selected by grid search, as
for the ensemble Kalman baselines. The reference-update smoothing parameters $\eta_1,\eta_2$
and the maximum iteration count $M$ are fixed in advance rather than tuned per case:
$\eta_1=\eta_2=0.5$ by default, reduced only if instability is observed, and $M$ is set based on
the convergence behavior in Figure~\ref{fig.linear.da.iter}, where the reference stabilizes within a few iterations.}

%% file: 5_numerical_new.tex
\section{Numerical Examples}\label{sec.numerical}
In this section, we present numerical experiments to demonstrate the performance of the proposed IEnSF method.

\subsection{Bias correction and non-Gaussianity for nonlinear Bayesian inference}\label{sec.bias_correction}
In this example, we evaluate the accuracy of the posterior distribution in terms of both bias correction and the ability of the posterior score to capture non-Gaussian behavior. To eliminate errors associated with constructing the GMM prior from finite ensembles, we assume a known Gaussian prior for all scenarios, with mean $[0,0]$ and covariance matrix 
$\begin{bmatrix}
0.5 & -0.4 \\
-0.4 & 0.5
\end{bmatrix}.$

\paragraph{Scenario 1: Linear observation model}
We first consider a linear observation operator $\mM(x_1,x_2) = x_1$ with observation noise standard deviation $\sigma_{\text{obs}} = 0.1$ and observed value $y=3$. In this case, the true posterior can be obtained analytically using the Kalman update formula, and this Gaussian posterior serves as the ground truth for comparison.

Figure~\ref{fig:linear-gaussian-iter} shows posterior samples obtained during the iterative refinement procedure of our method. The results demonstrate that the iterative scheme progressively improves the posterior approximation by updating the Gaussian posterior guess $\mN(\bmu^*, \bSigma^*)$, ultimately producing samples that closely match the true posterior. To quantify convergence, we compute the KL divergence between the posterior samples and the analytic Gaussian posterior. The evolution of the KL divergence, shown in Figure~\ref{fig:linear-gaussian-iter-KL}, confirms that the approximation improves with each iteration. 
 {
We emphasize that the improved score approximation and the iterative refinement are two separate mechanisms: the score expression already reduces the bias at the first pass, while the iteration corrects the reference posterior, which is initialized from the prior samples alone and therefore differs from the true posterior. The number of iterations is governed by the distance between the posterior and this prior-based reference, so when the two are close the refinement converges in very few iterations. Accordingly, the linear cases here serve as a consistency check that IEnSF recovers the correct Gaussian posterior, rather than a setting where IEnSF is expected to outperform the Kalman filter.
}

We also compare posterior samples from IEnSF and EnSF. The EnSF consistently exhibits bias with respect to the true posterior, visible both in the sample distribution (Figure~\ref{fig:linear-gaussian-iter-KL}, middle) and in the KL divergence (Figure~\ref{fig:linear-gaussian-iter-KL}, left). In contrast, our iterative method produces posterior samples that align much more closely with the analytic posterior, as reflected in the lower KL values and improved sample distribution (Figure~\ref{fig:linear-gaussian-iter-KL}, right). These results demonstrate that the proposed iterative posterior update, together with the refined posterior score approximation, substantially reduces bias in posterior sampling.

\begin{figure}[h]
    \centering
    \includegraphics[width=0.9\linewidth]{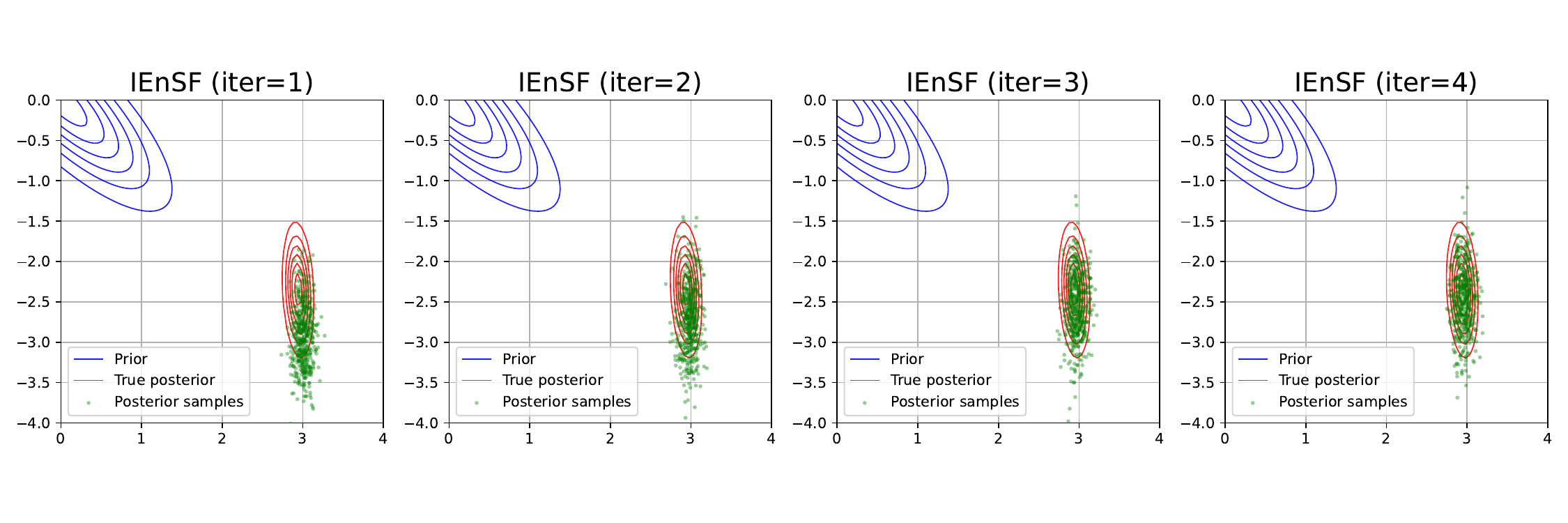}
    \caption{Posterior samples generated during the iterative refinement procedure for the linear observation model. The true posterior is Gaussian and obtained analytically from the Kalman update. Iterative updates progressively reduce the discrepancy between the samples and the true posterior.}
    \label{fig:linear-gaussian-iter}
\end{figure}

\begin{figure}[h]
    \centering
    \includegraphics[width=0.99\linewidth]{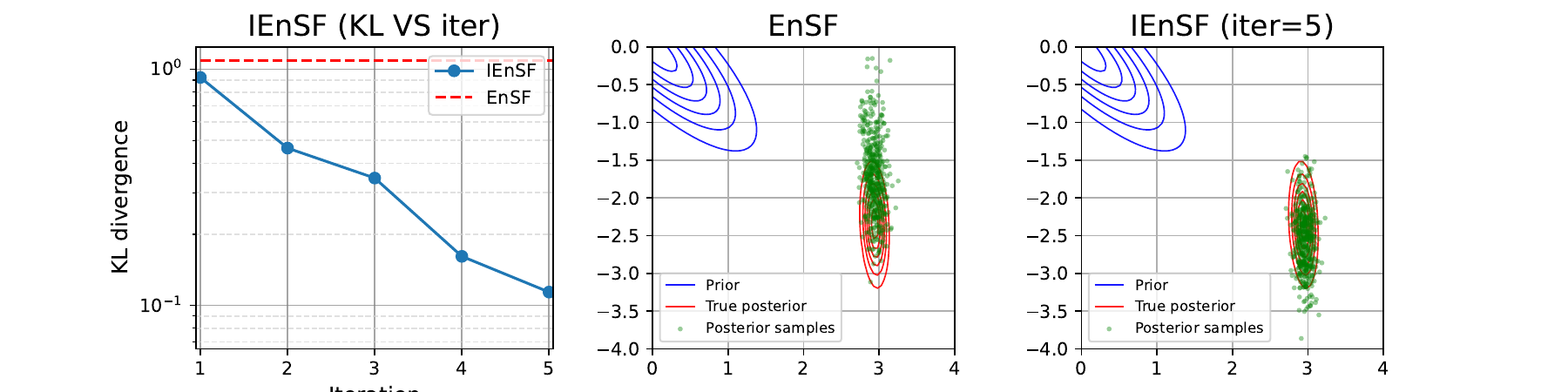}
    \caption{KL divergence and posterior samples for the linear observation model. 
    \textbf{Left:} KL divergence between posterior samples and the true Gaussian posterior across iterations. 
    \textbf{Middle:} Posterior samples from EnSF. 
    \textbf{Right:} Posterior samples from IEnSF at iteration 5. 
    The proposed IEnSF converges rapidly toward the true posterior, while EnSF retains a persistent bias.}
    \label{fig:linear-gaussian-iter-KL}
\end{figure}

\paragraph{Scenario 2: Non-Gaussian posterior}
We next examine the ability of our method to handle nonlinear observation operators, which can lead to highly non-Gaussian posteriors. Specifically, we consider the nonlinear operator 
$\mM(x_1, x_2) = \sqrt{(x_1 - 1)^2 + (x_2 - 1)^2}$
with observed value $y = 1.5$ and observation noise standard deviation $\sigma_{\text{obs}} = 0.1$.

Figure~\ref{fig:non-Gaussian-dist} shows the posterior ensembles obtained from IEnSF across the first four iterations, along with results from EnKF, the original EnSF, importance sampling with resampling (particle filter), and a computationally expensive long-run MCMC sampler, which we treat as the ground truth. The results demonstrate that IEnSF captures the non-Gaussian posterior effectively, even at the first iteration. This is because we initialize the posterior guess from the prior distribution, and in this example, the prior and posterior are already close. As a result, the initial posterior guess introduces only a small error, which the reverse SDE can effectively correct, yielding accurate posterior samples even without the iterative refinement. This highlights the robustness of the EnSF framework. The original EnSF also performs well in this setting, while the PF succeeds due to the absence of weight degeneracy. In contrast, EnKF fails to capture the posterior because of its restrictive Gaussian assumption.

\begin{figure}
    \centering
    \includegraphics[width=0.9\linewidth]{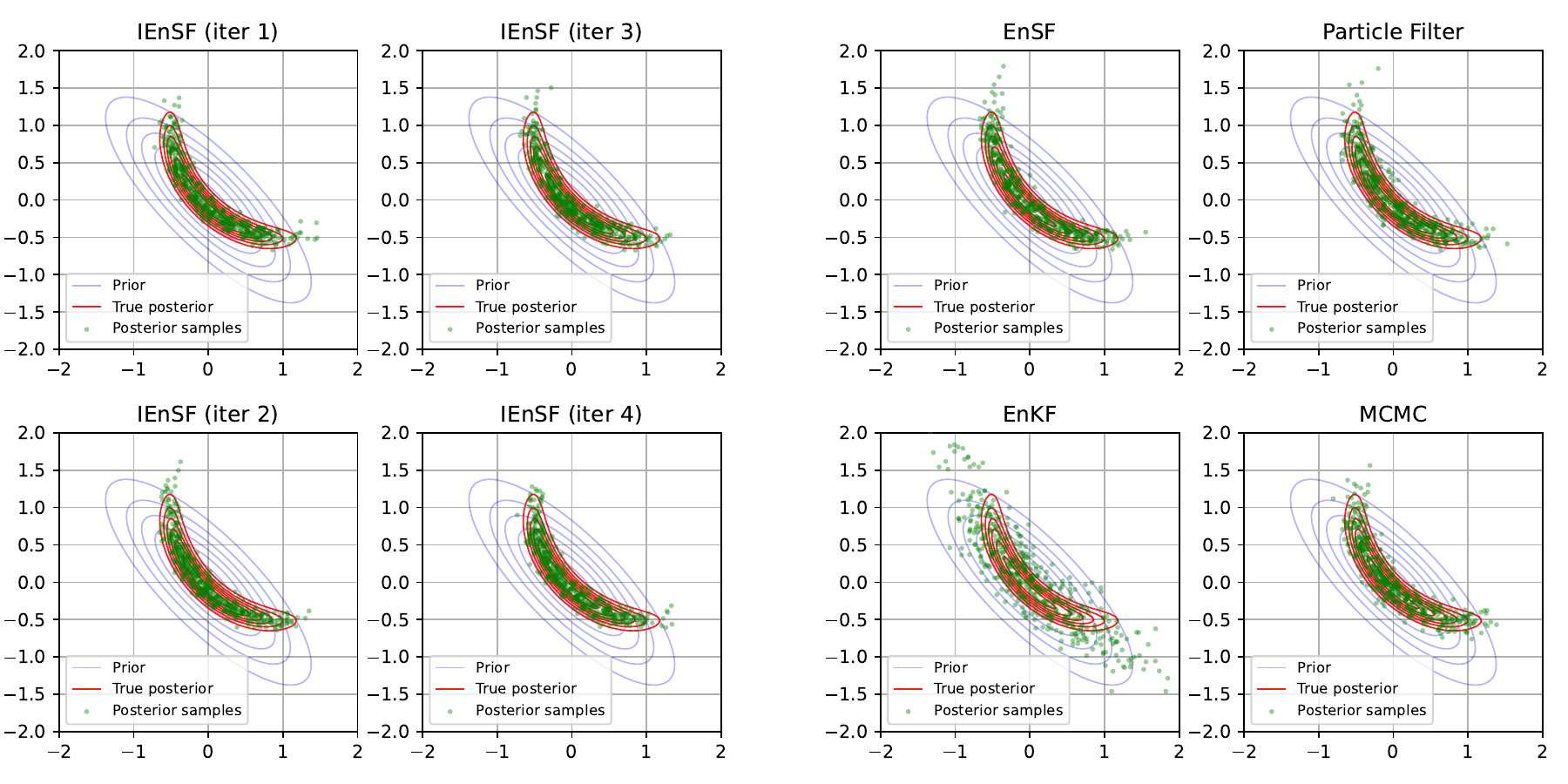}
    \caption{Non-Gaussian posterior sampling with a nonlinear observation operator. Results are shown for IEnSF (first four iterations), EnKF, particle filter, the original EnSF, and MCMC (ground truth). IEnSF captures the non-Gaussian posterior effectively even at the first iteration, since the initial reference posterior (based on the prior distribution) is already close to the true posterior.}
    \label{fig:non-Gaussian-dist}
\end{figure}

\paragraph{Scenario 3: Posterior bias correction}
Finally, we investigate posterior bias correction in settings where the posterior is far from the prior due to nonlinear observation operators. Such cases arise in practice under sudden exogenous shocks or in high-dimensional systems. Here we use
$\mM(x_1, x_2) = \sqrt{(x_1 -6)^2 + (x_2 -6)^2}$ 
with observed value $y = 2$ and $\sigma_{\text{obs}} = 0.3$.

Figure~\ref{fig:bias-correction-dist} illustrates that the particle filter suffers from severe weight degeneracy, with only a single particle surviving after resampling. EnKF broadly covers the posterior support but introduces substantial errors due to its Gaussian assumption. The original EnSF exhibits a systematic bias because its heuristic posterior score construction overemphasizes the likelihood term, pulling the posterior samples too close to the observation. Our method initially displays a similar bias, since the first posterior guess, which is based on the prior, is far from the true posterior. However, the iterative refinement progressively corrects this discrepancy, and the posterior samples converge toward the true distribution.

\begin{figure}[h]
    \centering
    \includegraphics[width=0.9\linewidth]{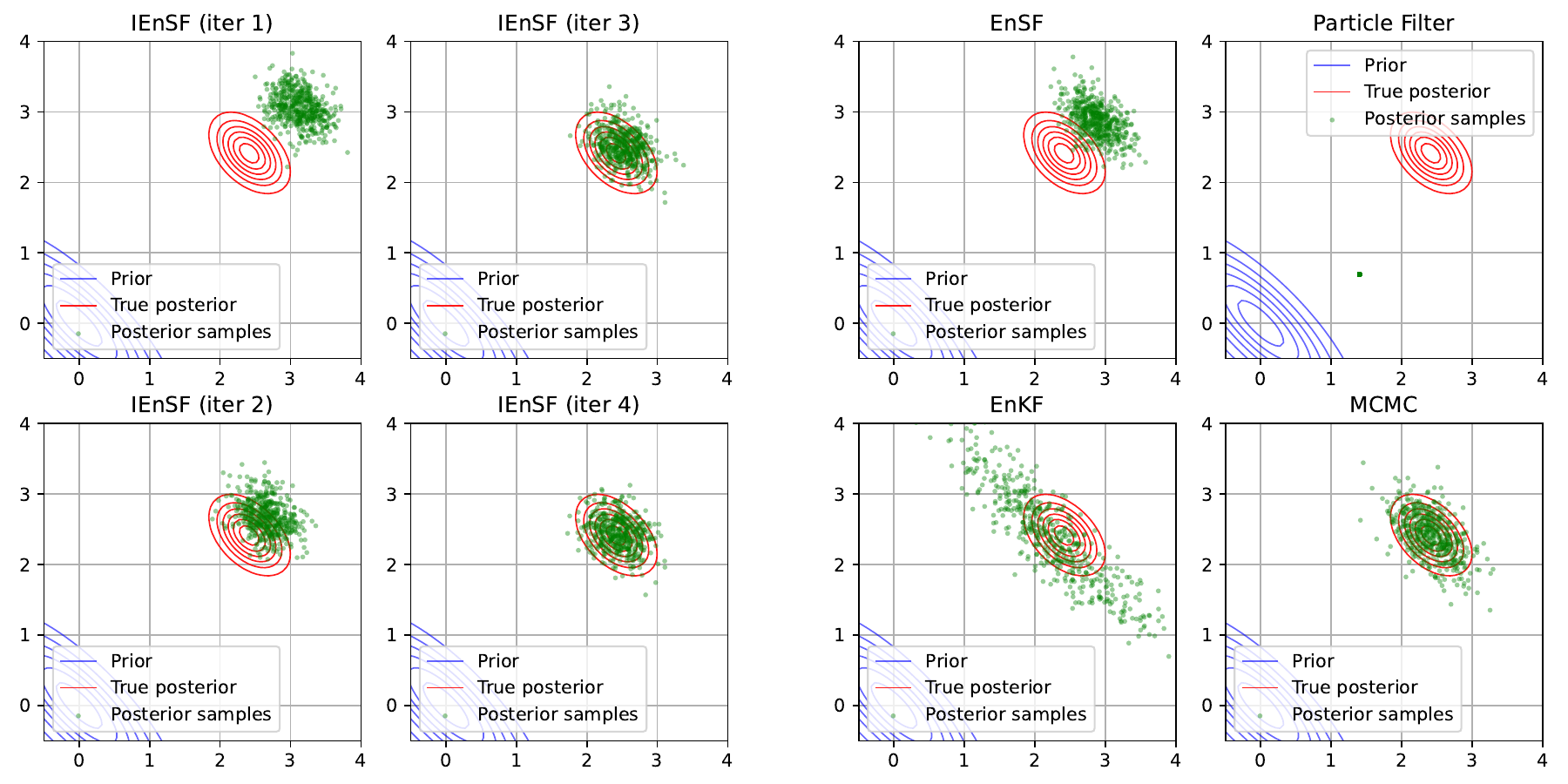}
    \caption{Bias correction in posterior sampling under a nonlinear observation operator. Compared methods include PF, EnKF, EnSF, and IEnSF. IEnSF reduces bias over iterations and converges toward the true posterior.}
    \label{fig:bias-correction-dist}
\end{figure}

\subsection{Harmonic oscillator}
We next evaluate the performance of IEnSF on a low-dimensional DA problem using a harmonic oscillator model.  
We consider a two-dimensional linear Gaussian state-space system corresponding to a discrete-time harmonic oscillator.  
Let $\bX_n \in \bbR^2$ denote the state at time step $n$, evolving according to  
\begin{equation}\label{eq.ho.state}
    \bX_{n+1} \;=\; 
    \begin{bmatrix}
        \cos(\omega \Delta t) & \tfrac{\sin(\omega \Delta t)}{\omega} \\[6pt]
        -\omega \sin(\omega \Delta t) & \cos(\omega \Delta t)
    \end{bmatrix} \bX_n + \bW_n, 
    \qquad \bW_n \sim \mathcal{N}(\bzero, \bQ).
\end{equation}
We set $\omega = 2.0$, $\Delta t = 0.1$, process noise covariance $\bQ = (0.5)^2 \bI_2$, and initial state $\bX_0 = [\,3.0,\,-3.0\,]^T$.  
The observation $\bY_n \in \bbR$ is obtained by directly measuring the first component of $\bX_n$ with Gaussian noise of standard deviation $\sigma_{\text{obs}} = 0.5$.  
The system is simulated for $100$ steps to generate reference state and observation sequences.  
All DA methods (EnKF, PF, EnSF, IEnSF) are initialized from a standard Gaussian.  
This linear setting serves as a controlled test to verify whether IEnSF can recover the true Gaussian posterior.  
Nonlinear observation cases are deferred to the Lorenz-96 example in Section~\ref{sec.l96}.  
 {As discussed in Section~\ref{sec.bias_correction}, the outer iteration in IEnSF refines only the Gaussian reference used to place the likelihood-score evaluation point, a correction that is needed even in linear problems because the reference is initialized from the prior alone; consistent with that discussion, Figure~\ref{fig.linear.da.iter} below shows this refinement converging within a few iterations in the present linear setting.}

For the linear case, the DA problem admits an exact solution via the Kalman filter (KF).  
Figure~\ref{fig.linead.DA.KL} compares the KL divergence of posterior ensembles generated by EnKF, PF, EnSF, and IEnSF against the true KF posterior.  
Each method uses 200 ensemble members, and results are averaged over 10 repetitions.  
For IEnSF, we use 5 refinement iterations.  
The EnKF achieves the smallest KL divergence since it makes a direct Gaussian assumption about the posterior distribution.  
Other sampling-based methods (PF, EnSF, IEnSF) naturally have larger KL values than EnKF.  
Compared with the consistent PF, IEnSF yields a very similar KL, only slightly larger due to approximation errors in the posterior score and numerical errors from solving the reverse SDE.  
By contrast, the original EnSF performs poorly, with a large KL divergence, due to the structural limitations of its posterior score construction. 

\begin{figure}[h!]
    \centering
    \includegraphics[width=0.7\linewidth]{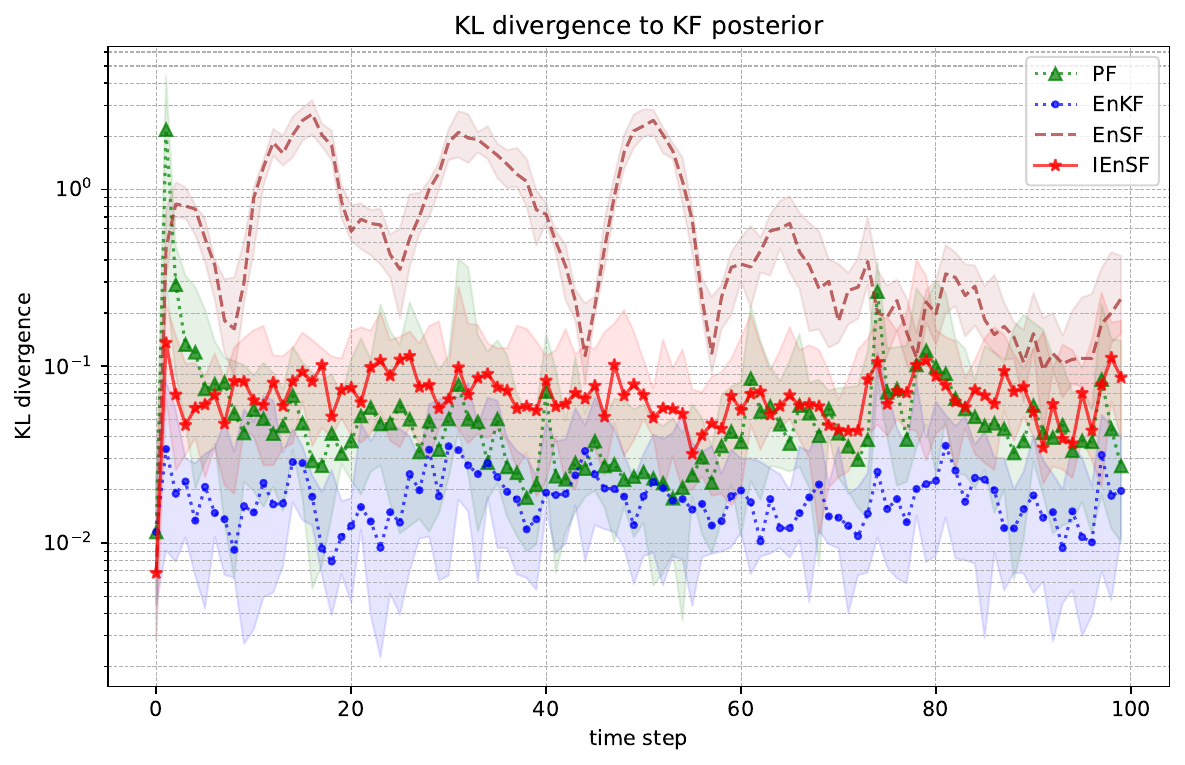}
    \caption{KL divergence of posterior ensembles against the KF posterior for the harmonic oscillator problem.  
    Each method uses 200 ensemble members with 10 repetitions. Error bars show the 10th and 90th percentiles. 
    EnKF performs best due to its Gaussian assumption, while IEnSF matches PF closely.  
    EnSF performs poorly due to limitations in its posterior score formulation.}
    \label{fig.linead.DA.KL}
\end{figure}

To further analyze the KL divergence results, Figure~\ref{fig.linear.DA.est} compares posterior means and spreads for each method along both state dimensions.  
The PF, EnKF, and IEnSF all produce posterior means consistent with the KF posterior, while EnSF exhibits noticeable bias in both dimensions, with especially large errors in the unobserved dimension.  
In terms of posterior spread, PF and EnKF agree with the true KF spread in the observed dimension, while IEnSF produces a slightly smaller spread, contributing to its marginally larger KL. 
In the unobserved dimension, all sampling-based methods (PF, EnSF, IEnSF) exhibit a similar spread, which initially increases and then stabilizes over time.

\begin{figure}[h!]
    \centering
    \includegraphics[width=0.99\linewidth]{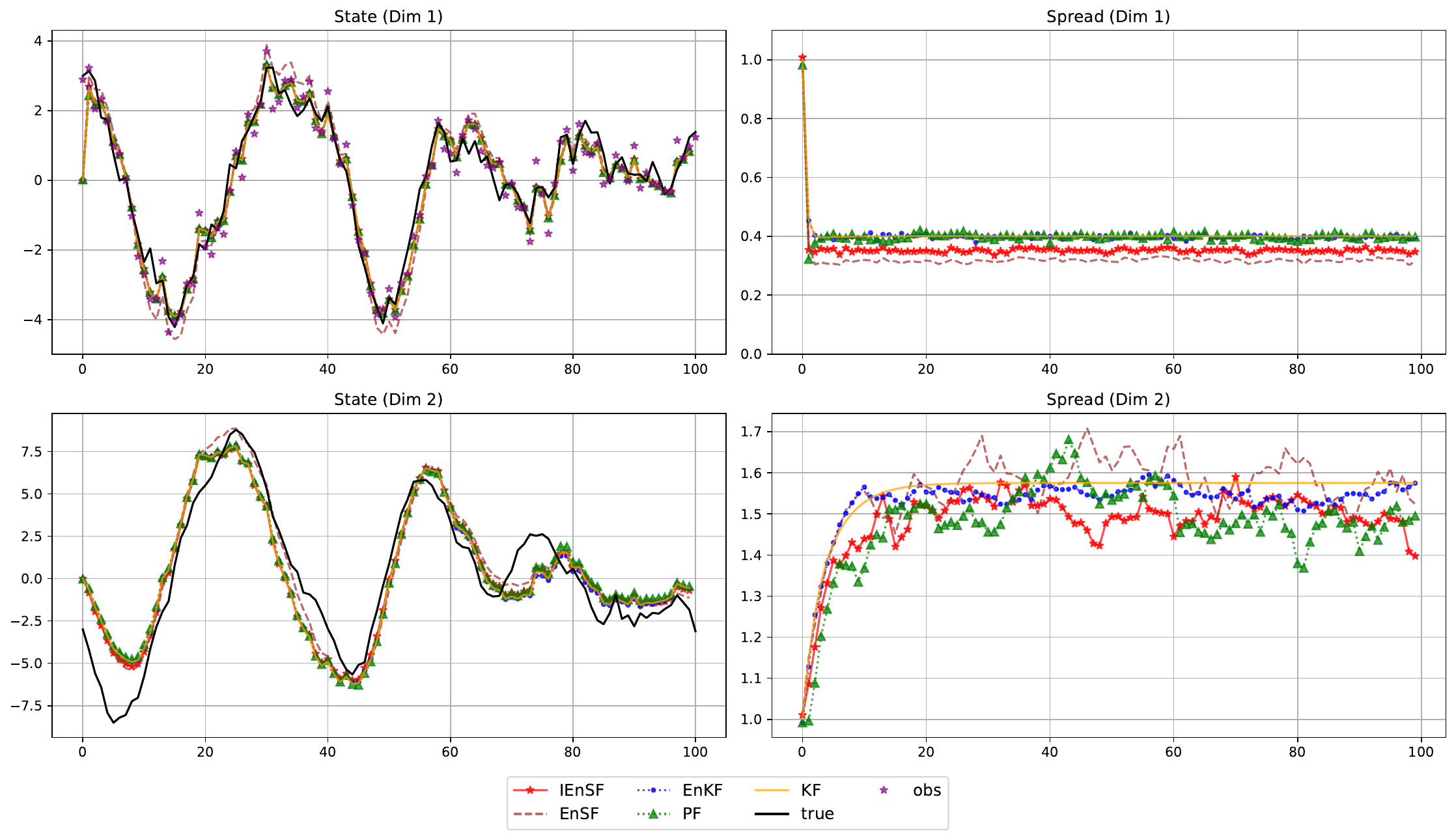}
    \caption{Posterior means and spreads for the two state dimensions of the harmonic oscillator.  
    PF, EnKF, and IEnSF agree closely with the KF posterior, while EnSF shows systematic bias and fails to track the unobserved dimension effectively.}
    \label{fig.linear.DA.est}
\end{figure}

Finally, to evaluate the effect of iterative refinement in IEnSF, Figure~\ref{fig.linear.da.iter} presents the average KL divergence between the IEnSF and KF posteriors as a function of iteration number.  
The KL divergence decreases rapidly during the first few iterations and stabilizes after about three, indicating convergence.  
Moreover, iterative refinement not only reduces the mean KL divergence but also narrows its spread, yielding more consistent and accurate approximations to the true posterior.  
These results highlight both the effectiveness and robustness of the proposed iterative refinement scheme.

\begin{figure}[h!]
    \centering
    \includegraphics[width=0.6\linewidth]{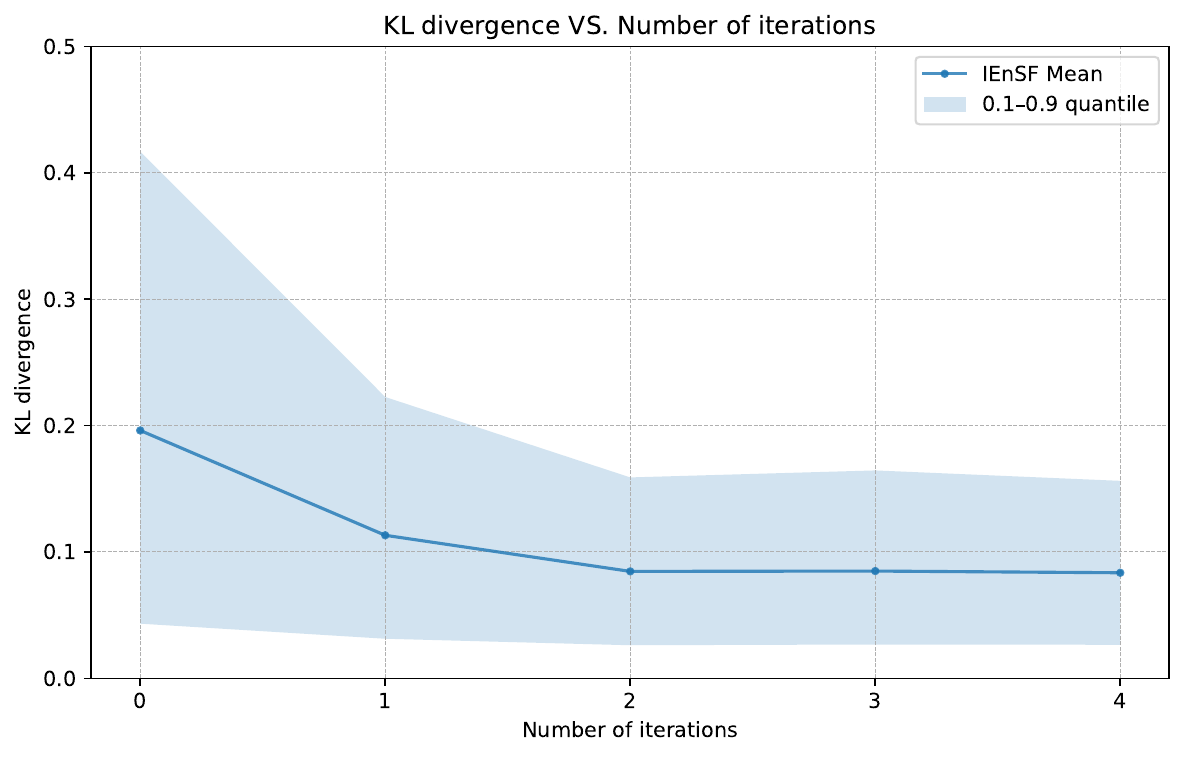}
    \caption{
    Average KL divergence between the IEnSF and KF posteriors as a function of iteration number. 
    The KL divergence is averaged over both time and repetitions, with shaded regions denoting the 0.1–0.9 quantiles. 
    Iterative refinement rapidly reduces both the mean and spread of the KL divergence, yielding more consistent approximations to the true posterior. 
    The KL divergence stabilizes after about three iterations, indicating convergence of the iterative refinement and demonstrating the robustness of the refinement.
    }
    \label{fig.linear.da.iter}
\end{figure}

\subsection{1000D Lorenz-96 (L96) model}\label{sec.l96}
We now evaluate the proposed method on a high-dimensional chaotic system using the 1000-dimensional Lorenz-96 (L96) model. Let $\bX(t) = [x_1(t), x_2(t), \cdots, x_d(t)]^{\top}$ with $d \geq 4$ denote the system state variable. The model dynamics are governed by the following ordinary differential equation:
\begin{equation}\label{eq:l96}
\frac{dx_i}{dt} = (x_{i+1} - x_{i-2}) x_{i-1} + F, 
\qquad i = 1, 2, \cdots, d, \quad d \geq 4,
\end{equation}
where periodic boundary conditions are applied, i.e., $x_{-1} \equiv x_{d-1}$, $x_{0} \equiv x_{d}$, and $x_{d+1} \equiv x_1$. We set $F = 8$ and $d = 1000$.

To generate the reference (ground-truth) trajectory, we integrate Eq.~\eqref{eq:l96} using a fourth-order Runge–Kutta (RK4) scheme with a time step of $\Delta t = 0.01$. The initial state $\bX(0)$ is drawn from a 1000-dimensional Gaussian distribution with mean 0 and standard deviation 3, then integrated for 1000 steps to produce a chaotic trajectory for the DA experiments.

We use 20 ensemble members for all DA methods (EnKF, LETKF, EnSF, and IEnSF).
 {The initial ensembles are randomly drawn (without replacement) from a precomputed pool of states on the L96 attractor, rather than perturbing the true state.}
For all DA methods, the forecast model is integrated with a coarser time step of $\Delta t = 0.02$, ensuring that none of the methods has direct access to the true dynamics. Figure~\ref{fig:l96-chaos} illustrates how the root mean square error (RMSE) of the L96 trajectory increases rapidly, even when initialized from the true state. This divergence results from the inherent chaotic nature of the system, where small numerical or initialization errors amplify exponentially and cause the forecast to deviate from the truth. This sensitivity to initial conditions highlights the importance of data assimilation for maintaining accurate and stable state estimates.

In addition, Figure~\ref{fig:l96-chaos} shows the reference RMSE of the forecast model without applying data assimilation, where the system is initialized with a perturbed true initial state and propagated forward using the coarse integration step. The forecast error grows much faster than in the case starting from the true initial state, until the simulated trajectory completely diverges from the truth, after which the RMSE stabilizes around 5. This result demonstrates how rapidly the chaotic dynamics of the L96 system can degrade predictive accuracy when data assimilation is not employed.

\begin{figure}[h!]
    \centering
    \includegraphics[width=0.6\linewidth]{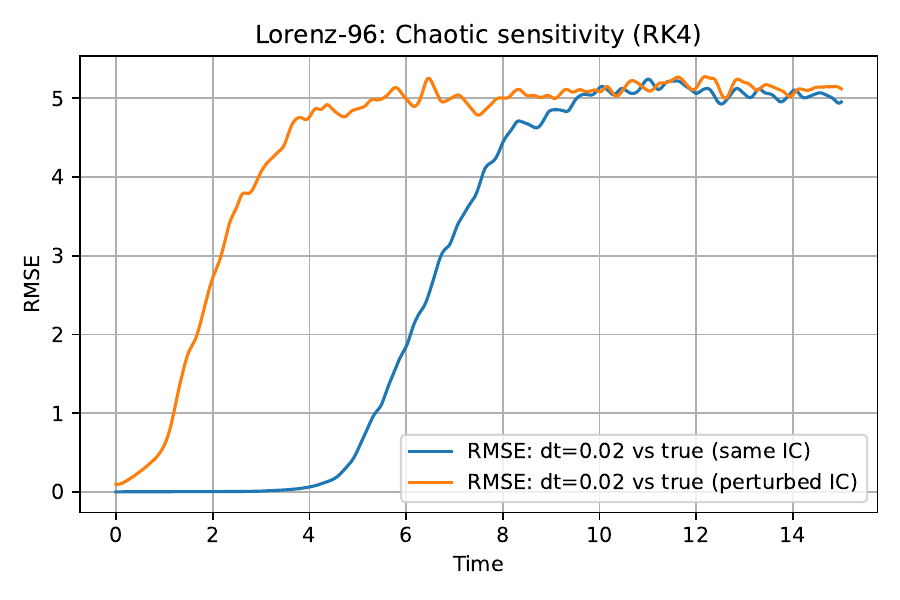}
    \caption{RMSE divergence of the Lorenz-96 trajectory over time. Even with the true initial condition, small numerical and initialization errors amplify due to chaotic dynamics, emphasizing the importance of data assimilation.}
    \label{fig:l96-chaos}
\end{figure}

We next consider two observation operators, direct and nonlinear, which differ in the fraction of the state observed.  {For the direct observation case, every fourth state variable is observed (25 percent of the system); for the $\arctan(\cdot)$ observation case, every second state variable is observed (50 percent of the system), since the denser network is needed to retain sufficient observational information once it is compressed by the nonlinear operator.} The observation noise standard deviation, $\sigma_{\mathrm{obs}}$, is set according to the specific observation type. Observation data are collected every $\Delta t = 0.2$, which roughly corresponds to one day (24 hours) in typical atmospheric models. For each DA method, we fine-tune hyperparameters such as inflation and covariance localization through grid search, using the Gaspari–Cohn function~\cite{gaspari1999construction} for localization.  {For the direct observation case: LETKF uses inflation 1.15 and localization radius 2; IEnSF uses $\gamma=1.0$, localization radius 2, inflation 0.9, and $M=4$; EnSF uses inflation 1.0. For the $\arctan(\cdot)$ observation case: LETKF uses inflation 1.3 and localization radius 2; IEnSF uses $\gamma=0.25$, a more particle-like prior appropriate for the stronger non-Gaussianity of this case, localization radius 3, inflation 0.8, a diagonal-likelihood blend of 0.3, and $M=4$; EnSF uses inflation 1.0. For IEnSF, inflation scales the spread of the Gaussian-mixture component means rather than the forecast covariance directly, so it is not directly comparable to the LETKF inflation factor, and values below one do not correspond to ensemble deflation in the LETKF sense.}  {IEnSF uses $M = 4$ refinement iterations, consistent with the convergence observed in Figure~\ref{fig.linear.da.iter}. The dominant per-cycle cost is one eigendecomposition of the localized covariance, reused across the $N_t$ reverse-SDE steps within each iteration; IEnSF's total per-cycle cost is therefore $O(M \cdot N_t)$ score evaluations plus one covariance eigendecomposition, in contrast to LETKF's per-cycle cost of $O(d)$ local $K^3$ solves.}

\paragraph{Direct observation model}
 {We first consider the direct observation case, where $m(x) = x$ for the observed
components, with observation noise standard deviation $\sigma_{\mathrm{obs}} = 0.1$. In
addition to LETKF and IEnSF, we include the original EnSF as a baseline, together with an
ablation, IEnSF with the improved score formula but without the iterative refinement
(``IEnSF, no iteration''), which isolates the contribution of the score-formula correction
from that of the iteration. As shown in Figure~\ref{fig.l96.lin.rmse.obs}, starting from the
poor, attractor-drawn ensemble, LETKF and both IEnSF variants correct the observed dimensions
to an RMSE consistent with $\sigma_{\mathrm{obs}}$ within a few assimilation cycles. On the
unobserved dimensions, the original EnSF remains near its initial, climatological error
throughout the run, since it lacks the covariance structure needed to propagate observed
information to unobserved components. IEnSF without iteration already substantially reduces
this error using the improved score formula alone, and the full IEnSF with iterative
refinement further reduces it, closely tracking LETKF; both LETKF and full IEnSF remain
stable over the full duration of the run.}

\begin{figure}[h!]
    \centering
    \includegraphics[width=0.9\linewidth]{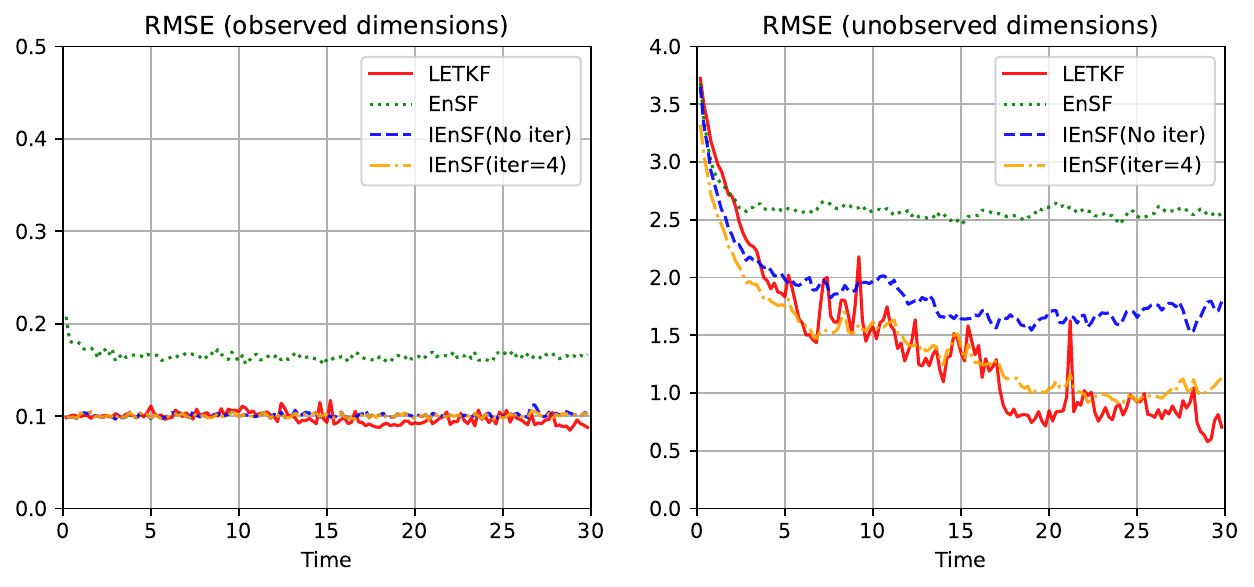}
    \caption{
     {RMSE of observed and unobserved components for the direct observation model, starting
    from a poor, attractor-drawn ensemble and averaged over 10 repetitions. LETKF and IEnSF
    (with iterative refinement) perform comparably, effectively updating both observed and
    unobserved state dimensions through cross-variable correlations, while the original EnSF
    shows little improvement on the unobserved dimensions. IEnSF without iterative refinement
    substantially improves on EnSF via the score-formula correction alone, with the iteration
    providing a further improvement.}
    }
    \label{fig.l96.lin.rmse.obs}
\end{figure}

\paragraph{$\arctan(\cdot)$ observation model}  
Next, we consider the nonlinear observation function $m(x) = \arctan(x)$ with $\sigma_{\mathrm{obs}} = 0.05$. This nonlinear observation model poses a much greater challenge for data assimilation than the direct observation case: the $\arctan$ function compresses large-magnitude values, effectively reducing the sensitivity of observations to state variations in those regions.

 {As in the direct observation case, we start from the poor, attractor-drawn ensemble, and in addition to LETKF and IEnSF we include the original EnSF as a baseline and the ``IEnSF, no iteration'' ablation. Both LETKF and IEnSF are re-tuned for this nonlinear setting; with this re-tuning, both methods remain stable over the full duration of the run. As shown in Figure~\ref{fig.l96.atan.rmse.obs}, IEnSF, with or without the iterative refinement, achieves substantially lower RMSE than both LETKF and EnSF on both the observed and unobserved dimensions, confirming the robustness to nonlinear observation operators already suggested in the direct observation case. As in that case, most of this improvement is already present without the iterative refinement, with the iteration providing a further, smaller reduction in the unobserved-dimension error.}

 {Because the Lorenz-96 system has a short spatial correlation length, it does not stress the many-correlated-variables-within-a-localization-radius regime of finely discretized PDEs such as quasi-geostrophic flows; extending the covariance-based framework to such systems requires physical-field localization and low-rank handling, which we leave to future work.}

\begin{figure}[h!]
    \centering
    \includegraphics[width=0.8\linewidth]{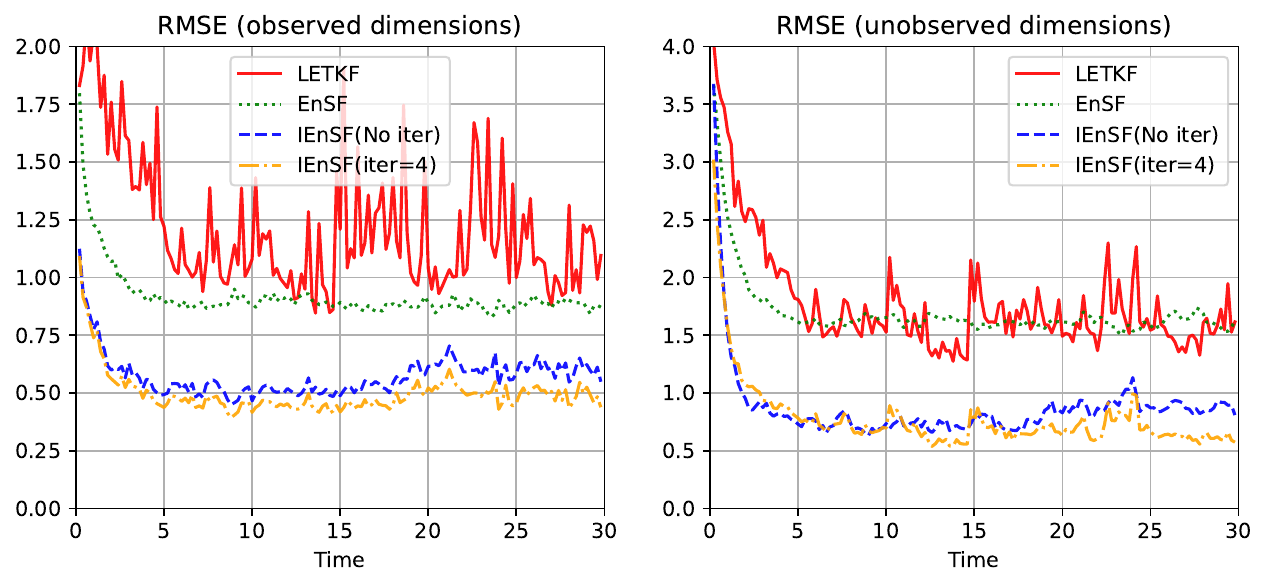}
    \caption{ {RMSE of observed and unobserved components for the nonlinear $\arctan(x)$ observation
    model, starting from a poor, attractor-drawn ensemble and averaged over 10 repetitions. IEnSF,
    with or without iterative refinement, achieves substantially lower RMSE than both LETKF and the
    original EnSF on both observed and unobserved dimensions, with the iteration providing a
    further improvement over the score-formula correction alone.}}
    \label{fig.l96.atan.rmse.obs}
\end{figure}

%% file: 6_conclusoin.tex
\section{Conclusion}\label{sec.conclusion}

In this work, we proposed the Iterative Ensemble Score Filter (IEnSF), an extension of the ensemble score filter (EnSF) designed to reduce structural errors in posterior score approximation.  
By deriving the exact posterior score expression under a Gaussian mixture prior assumption, we identified the key components required for consistent score estimation.  
We then developed practical approximation schemes for the weighting function and likelihood score, and introduced an iterative refinement strategy to correct biases during the reverse SDE sampling.  
This iterative update effectively incorporates observational information while maintaining stability in the posterior distribution.

Our numerical experiments demonstrated the advantages of IEnSF across a range of scenarios.  
In low-dimensional linear settings, IEnSF achieved posterior consistency comparable to the Kalman filter and particle filter, significantly outperforming the original EnSF.  
In nonlinear cases with non-Gaussian posteriors, IEnSF reduced bias and captured non-Gaussian features effectively, even under sparse observation operators.  
Finally, in the high-dimensional Lorenz-96 model, IEnSF achieved performance competitive with state-of-the-art methods such as LETKF, while retaining flexibility in handling nonlinearity.  
These results highlight the potential of diffusion-based ensemble methods for scalable and accurate Bayesian filtering.  
Future directions include extending IEnSF to joint state–parameter estimation, integrating adaptive schemes for likelihood score refinement,  {optimizing the initialization of the reference posterior, for example through a Kalman-family update, to accelerate the convergence of the iterative refinement,} and applying the method to large-scale geophysical models.  
More broadly, the iterative refinement framework may also be useful for other score-based generative approaches in uncertainty quantification and data assimilation.

%% file: 9_appendix.tex
\newpage
\section{Derivation of Propositions \ref{prop.prior.trans.density} and \ref{prop.gmm.prior.score}}\label{app.prop.derivation}
In this section, we derive Propositions~\ref{prop.prior.trans.density} and~\ref{prop.gmm.prior.score}. 
We begin with the following lemma on the distributional form of the solution to the forward SDE in Eq.~\eqref{eq.forward.sde}.
\begin{lemma}[Distribution of the forward SDE]\label{lemma:zt-dist}
If $\bZ_t$ follows the forward SDE in Eq.~\eqref{eq.forward.sde}, then the marginal distribution of $\bZ_t$ is
\begin{equation}\label{eq:zt-dist}
    \bZ_t \deq \alpha_t \bZ_0 + \beta_t \bveps,
\end{equation}
where $\bveps \sim \mathcal{N}(\bzero, \bI)$ is independent of $\bZ_0$.
\end{lemma}
\begin{proof}
Since the forward SDE in Eq.~\eqref{eq.forward.sde} is a linear SDE, the process $\{\bZ_t\}_t$ admits a strong solution given by (see \cite{oksendal2013stochastic})
\begin{equation}\label{eq.forward.sde.solution}
    \bZ_t = \bZ_0 \alpha(t) + \alpha(t) \int_0^t \frac{b(s)}{\alpha(s)} d\bW_s,
\end{equation}
where $\bW_s$ is a standard Brownian motion independent of $\bZ_0$.  
The Itô integral term $\int_0^t \frac{b(s)}{\alpha(s)} d\bW_s$ follows a Gaussian distribution with
\begin{align}
    \mathbb{E}\!\left[\int_0^t \frac{b(s)}{\alpha(s)} d\bW_s \right] &= \bzero,\\
    \mathrm{Cov}\!\left[\int_0^t \frac{b(s)}{\alpha(s)} d\bW_s \right] &= \frac{\beta^2(t)}{\alpha^2(t)} \bI.
\end{align}
Therefore, the marginal distribution of $\bZ_t$ can be expressed as
\begin{equation}
    \bZ_t \deq \alpha_t \bZ_0 + \beta_t \bveps,
\end{equation}
where $\bveps \sim \mathcal{N}(\bzero, \bI)$ is independent of $\bZ_0$.  
This establishes the desired distribution of $\bZ_t$.
\end{proof}

This result also gives the forward transitional distribution as follows
\begin{equation}
    p_{\bZ_t | \bZ_0}(\bz_t | \bz_0) = \mN(\alpha_t \bz_0, \, \beta_t^2 \bI_d).
\end{equation}
We next present two lemmas that facilitate the computation of the score function.

\begin{lemma}[Gradient of exponentiated functions]\label{lemma:exp-fun-grad}
For a function of the form $g(\bx) = c \exp(f(\bx))$, where $c$ is a constant, we have
\begin{equation}
    \nabla_{\bx} g(\bx) = g(\bx) \nabla_{\bx} \log g(\bx) = g(\bx) \bS_g(\bx),
\end{equation}
where $\bS_g(\bx) = \nabla_{\bx} \log g(\bx)$ is the score function of $g(\bx)$.
\end{lemma}
\begin{proof}
\begin{align*}
    \nabla_{\bx} g(\bx) &= c \nabla_{\bx} \exp(f(\bx)) = c \exp(f(\bx)) \nabla_{\bx} f(\bx) = g(\bx) \nabla_{\bx} \log g(\bx).
\end{align*}
\end{proof}
\begin{lemma}[Gaussian score function]\label{lemma:gaussian-score}
Let $\bX \sim \mathcal{N}(\bmu, \bSigma)$. Then the Gaussian score function 
$\bS_{\mathcal{N}}(\bx; \bmu, \bSigma) \equiv \nabla_{\bx} \log p_{\bX}(\bx)$ 
is given by
\begin{equation}
    \bS_{\mathcal{N}}(\bx; \bmu, \bSigma) = -\bSigma^{-1} (\bx - \bmu).
\end{equation}
\end{lemma}
\begin{proof}
By definition, 
\[
    p_{\bX}(\bx) = c \exp\!\left(-\tfrac{1}{2}(\bx - \bmu)^\top \bSigma^{-1} (\bx - \bmu)\right),
\]
where $c$ is a normalizing constant. Taking the gradient gives
\begin{align*}
    \nabla_{\bx} \log p_{\bX}(\bx) 
    &= \nabla_{\bx} \left[\log c - \tfrac{1}{2} (\bx - \bmu)^\top \bSigma^{-1} (\bx - \bmu)\right] = -\bSigma^{-1} (\bx - \bmu).
\end{align*}
\end{proof}

We next establish the form of the transitional kernel when the initial distribution is a single Gaussian.  
The following lemma shows that the forward SDE preserves Gaussianity over time and that both the forward and reverse conditional distributions admit closed-form expressions.

\begin{lemma}[Single Gaussian transitional kernel]\label{lemma.single.gaussian.trans}
Let $\bZ_t$ follow the forward SDE in Eq.~\eqref{eq.forward.sde}, and suppose the initial distribution is $\bZ_0 \sim \mathcal{N}(\bmu, \bSigma)$.  
Then the marginal distribution at time $t$ is Gaussian:
\begin{equation}
    p_{\bZ_t}(\bz_t) = \mathcal{N}(\bmu_t, \bSigma_t),
\end{equation}
where the forward mean and covariance are given by
\begin{align}
    \bmu_t &= \alpha_t \bmu, \\
    \bSigma_t &= \alpha_t^2 \bSigma + \beta_t^2 \bI.
\end{align}

Furthermore, the corresponding reverse (conditional) distribution is also Gaussian:
\begin{equation}
    p_{\bZ_0 | \bZ_t}(\bz_0 | \bz_t) = \mathcal{N}\big(\bmu_{0|t}(\bz_t), \bSigma_{0|t}\big),
\end{equation}
where the conditional mean and covariance are given by
\begin{align}
    \bmu_{0|t}(\bz_t) &= \bmu 
    + \alpha_t \bSigma \left(\alpha_t^2 \bSigma + \beta_t^2 \bI\right)^{-1} 
    (\bz_t - \alpha_t \bmu), \\
    \bSigma_{0|t} &= 
    \bSigma - \alpha_t^2 \bSigma \left(\alpha_t^2 \bSigma + \beta_t^2 \bI\right)^{-1} \bSigma.
\end{align}
\end{lemma}
\begin{proof}
From Lemma~\ref{lemma:zt-dist}, the marginal distribution of $\bZ_t$ satisfies
\begin{equation}\label{eq.ztz0.relation.adpx}
    \bZ_t \deq \alpha_t \bZ_0 + \beta_t \bveps,
\end{equation}
where $\bveps \sim \mathcal{N}(\bzero, \bI_d)$ is independent of $\bZ_0$.  
Given that $\bZ_0 \sim \mathcal{N}(\bmu, \bSigma)$, it follows directly that $\bZ_t$ is also Gaussian with mean and covariance
\begin{align}
    \bmu_t &= \alpha_t \bmu, \\
    \bSigma_t &= \alpha_t^2 \bSigma + \beta_t^2 \bI_d.
\end{align}

To derive $p_{\bZ_0 | \bZ_t}(\bz_0 | \bz_t)$, we first note that $\bZ_0$ is Gaussian and $\bZ_t$ is linearly related to $\bZ_0$ through Eq.~\eqref{eq.ztz0.relation.adpx}.  
By the properties of linear transformations of Gaussian variables, it follows that both the joint distribution $p_{\bZ_0, \bZ_t}(\bz_0, \bz_t)$ and the conditional distribution $p_{\bZ_0 | \bZ_t}(\bz_0 | \bz_t)$ are Gaussian.  
To make this explicit, we express the joint relationship between $\bZ_0$ and $\bZ_t$ as
\begin{equation}\label{eq:joint-eq-z0-zt}
\begin{bmatrix}
\bZ_0\\[3pt]
\bZ_t
\end{bmatrix}
=
\begin{bmatrix}
\bI_d & \bzero\\[3pt]
\alpha_t \bI_d & \beta_t \bI_d
\end{bmatrix}
\begin{bmatrix}
\bZ_0\\[3pt]
\bveps
\end{bmatrix},
\end{equation}
where $\bI_d \in \mathbb{R}^{d\times d}$ is the identity matrix and $\bzero \in \mathbb{R}^{d\times d}$ is the zero matrix.  
From Lemma~\ref{lemma:zt-dist}, we have $\bveps \sim \mathcal{N}(\bzero, \bI_d)$ and is independent of $\bZ_0$.  
Hence, the joint distribution of $[\bZ_0, \bveps]^\top$ is
\begin{equation}\label{eq:z0.eps.joint.gaussian}
\begin{bmatrix}
\bZ_0\\[3pt]
\bveps
\end{bmatrix}
\sim
\mathcal{N}\!\left(
\begin{bmatrix}
\bmu\\[3pt]
\bzero
\end{bmatrix},
\begin{bmatrix}
\bSigma & \bzero\\[3pt]
\bzero & \bI_d
\end{bmatrix}
\right).
\end{equation}
Equation~\eqref{eq.ztz0.relation.adpx} defines a linear transformation of the jointly Gaussian vector $[\bZ_0, \bveps]^\top$ in Eq.~\eqref{eq:z0.eps.joint.gaussian}.  
Therefore, $[\bZ_0, \bZ_t]^\top$ is also jointly Gaussian with
\begin{equation}
\begin{bmatrix}
\bZ_0\\[3pt]
\bZ_t
\end{bmatrix}
\sim
\mathcal{N}\!\left(
\begin{bmatrix}
\bmu\\[3pt]
\alpha_t \bmu
\end{bmatrix},
\begin{bmatrix}
\bSigma & \alpha_t \bSigma\\[3pt]
\alpha_t \bSigma & \alpha_t^2 \bSigma + \beta_t^2 \bI_d
\end{bmatrix}
\right).
\end{equation}

Given this joint Gaussian distribution, the conditional distribution $p_{\bZ_0 | \bZ_t}(\bz_0 | \bz_t)$ is also Gaussian:
\begin{equation}
    p_{\bZ_0 | \bZ_t}(\bz_0 | \bz_t) = \mathcal{N}\!\big( \bmu_{0|t}(\bz_t), \bSigma_{0|t} \big),
\end{equation}
where
\begin{align}
    \bmu_{0|t}(\bz_t) &= \bmu + \alpha_t \bSigma \left( \alpha_t^2 \bSigma + \beta_t^2 \bI_d \right)^{-1} (\bz_t - \alpha_t \bmu), \\
    \bSigma_{0|t} &= \bSigma - \alpha_t^2 \bSigma \left( \alpha_t^2 \bSigma + \beta_t^2 \bI_d \right)^{-1} \bSigma.
\end{align}
\end{proof}

\noindent Next, we can derive Proposition~\ref{prop.prior.trans.density}.
\begin{proof}
From the assumption on the initial Gaussian mixture prior, we have
\begin{equation}\label{eq.gmm.initial.apdx}
    p_{\bZ_0}(\bz_0) = \sum_{k=1}^K \pi_k \, \phi(\bz_0; \bmu_k, \bSigma),
\end{equation}
where $\phi(\cdot; \bmu_k, \bSigma)$ denotes the Gaussian density with mean $\bmu_k$ and covariance $\bSigma$, and $\pi_k$ is the mixture weight of component $k$.  
Let $\xi \in \{1, \dots, K\}$ denote the latent variable indicating the Gaussian mixture weight.

For the forward marginal density, by the law of total probability, we can write
\begin{align}
    p_{\bZ_t}(\bz_t) 
    &= \sum_{k=1}^K p_{\bZ_t | \xi}(\bz_t | k) \, p_{\xi}(k) = \sum_{k=1}^K p_{\xi}(k) \int p_{\bZ_t | \bZ_0}(\bz_t | \bz_0) \, p_{\bZ_0 | \xi}(\bz_0 | k) \, d\bz_0. \label{eq.forward.mixture.marginal}
\end{align}
Since conditioning on $\xi = k$ fixes the initial component $p_{\bZ_0 | \xi}(\bz_0 | k)$ as Gaussian $\mN(\bmu_k, \bSigma)$, and the transition density $p_{\bZ_t | \bZ_0}(\bz_t | \bz_0)$ of the forward SDE does not depend on $\xi$, the marginal $p_{\bZ_t | \xi}(\bz_t | k)$ corresponds to the forward SDE marginal with $\mN(\bmu_k, \bSigma)$ as initial condition.  
From Lemma~\ref{lemma.single.gaussian.trans}, we therefore obtain
\begin{equation}
    p_{\bZ_t}(\bz_t) = \sum_{k=1}^K p_{\xi}(k)\, \phi(\bz_t; \bmu_{t,k}, \bSigma_{t}),
\end{equation}
where each Gaussian component evolves according to
\begin{align}
    \bmu_{t,k} &= \alpha_t \bmu_k, \label{eq.gmm.forward.mean.apdx}\\
    \bSigma_{t} &= \alpha_t^2 \bSigma + \beta_t^2 \bI_d. \label{eq.gmm.forward.cov.apdx}
\end{align}

For the reverse kernel $p_{\bZ_0 | \bZ_t}(\bz_0 | \bz_t)$, we apply a similar conditioning argument using $\xi$ as the latent variable:
\begin{align}
    p_{\bZ_0 | \bZ_t}(\bz_0 | \bz_t) 
    &= \sum_{k=1}^K p_{\bZ_0 | \bZ_t, \xi}(\bz_0 | \bz_t, k) \, p_{\xi | \bZ_t}(k | \bz_t). \label{eq.reverse.kernel.mixture}
\end{align}

Conditioned on $\xi = k$, the pair $(\bZ_0, \bZ_t)$ follows a joint Gaussian distribution, and by Lemma~\ref{lemma.single.gaussian.trans} the conditional density takes the form
\begin{equation}
    p_{\bZ_0 | \bZ_t, \xi}(\bz_0 | \bz_t, k)
    = \phi(\bz_0; \bmu_{0|t,k}(\bz_t), \bSigma_{0|t}),
\end{equation}
where
\begin{align}
    \bmu_{0|t,k}(\bz_t) &= \bmu_k + \alpha_t \bSigma \big( \alpha_t^2 \bSigma + \beta_t^2 \bI_d \big)^{-1} (\bz_t - \alpha_t \bmu_k), \\
    \bSigma_{0|t} &= \bSigma - \alpha_t^2 \bSigma \big( \alpha_t^2 \bSigma + \beta_t^2 \bI_d \big)^{-1} \bSigma.
\end{align}

The posterior mixture weights are obtained by applying Bayes' rule:
\begin{align}
    p_{\xi | \bZ_t}(k | \bz_t) 
    &= \frac{p_{\bZ_t | \xi}(\bz_t | k)\, p_{\xi}(k)}{p_{\bZ_t}(\bz_t)} 
    = \frac{\pi_k \, \phi(\bz_t; \bmu_{t,k}, \bSigma_{t})}
    {\sum_{j=1}^K \pi_j \, \phi(\bz_t; \bmu_{t,j}, \bSigma_{t})}. \label{eq.p.xi.zt.apdx}
\end{align}

Combining Eqs.~\eqref{eq.reverse.kernel.mixture}–\eqref{eq.p.xi.zt.apdx} yields the desired Gaussian mixture form of the reverse transition kernel, completing the proof.
\end{proof}

\noindent Next, we derive the score function of the prior process $\bZ_t$ in Proposition~\ref{prop.gmm.prior.score}.

\begin{proof}
From Proposition~\ref{prop.prior.trans.density}, the marginal distribution of $\bZ_t$ is a Gaussian mixture given by
\begin{equation}
    p_{\bZ_t}(\bz_t) = \sum_{k=1}^K p_{\xi}(k)\, \phi(\bz_t; \bmu_{t,k}, \bSigma_{t}),
\end{equation}
where each component has mean $\bmu_{t,k} = \alpha_t \bmu_k$ and covariance $\bSigma_{t} = \alpha_t^2 \bSigma + \beta_t^2 \bI_d$.

By definition, the score function of $p_{\bZ_t}(\bz_t)$ is
\begin{equation}\label{eq.prior.score.def.apdx}
    \bS_{\bZ_t}(\bz_t) 
    = \nabla_{\bz_t} \log p_{\bZ_t}(\bz_t)
    = \frac{\nabla_{\bz_t} p_{\bZ_t}(\bz_t)}{p_{\bZ_t}(\bz_t)}.
\end{equation}

We first compute the gradient of the mixture density:
\begin{align}
    \nabla_{\bz_t} p_{\bZ_t}(\bz_t)
    &= \nabla_{\bz_t} \sum_{k=1}^K p_{\xi}(k)\, \phi(\bz_t; \bmu_{t,k}, \bSigma_{t}) = \sum_{k=1}^K p_{\xi}(k)\, \nabla_{\bz_t} \phi(\bz_t; \bmu_{t,k}, \bSigma_{t}) \nonumber\\
    &= \sum_{k=1}^K p_{\xi}(k)\, \phi(\bz_t; \bmu_{t,k}, \bSigma_{t})\, \bS_{\bZ_t | k}(\bz_t), \label{eq.grad.mix}
\end{align}
where the last equality follows from Lemma~\ref{lemma:exp-fun-grad}, since the Gaussian PDF belongs to the exponential family.  
Here, $\bS_{\bZ_t | k}(\bz_t)$ denotes the score of the $k$-th Gaussian component,
\begin{equation}
    \bS_{\bZ_t | k}(\bz_t) = -\bSigma_t^{-1}(\bz_t - \bmu_{t,k}).
\end{equation}

Substituting Eq.~\eqref{eq.grad.mix} into Eq.~\eqref{eq.prior.score.def.apdx}, we obtain
\begin{align}
    \bS_{\bZ_t}(\bz_t)
    &= \frac{\sum_{k=1}^K p_{\xi}(k)\, \phi(\bz_t; \bmu_{t,k}, \bSigma_{t})\, \bS_{\bZ_t | k}(\bz_t)}
    {p_{\bZ_t}(\bz_t)} 
    = \sum_{k=1}^K \frac{p_{\xi}(k)\, p_{\bZ_t | \xi}(\bz_t | k)}
    {p_{\bZ_t}(\bz_t)}\, \bS_{\bZ_t | k}(\bz_t) \nonumber\\
    &= \sum_{k=1}^K p_{\xi | \bZ_t}(k | \bz_t)\, \bS_{\bZ_t | k}(\bz_t), \label{eq.gmm.score.final}
\end{align}
where the posterior mixture weights $p_{\xi | \bZ_t}(k | \bz_t)$ are given by Eq.~\eqref{eq.p.xi.zt.apdx}.  

Thus, the score function of the Gaussian mixture prior process $\bZ_t$ is a weighted average of the individual component scores, with weights given by $p_{\xi | \bZ_t}(k | \bz_t)$.
\end{proof}

\section{Derivation of the posterior score}\label{app.post.score.derivation}
We now derive the posterior score function stated in Theorem~\ref{thm.post.score.expression} under Assumption~\ref{assump.gmm}.  
By definition, the posterior score function corresponds to the score of the forward diffusion process $\bZ_t^{\bY}$ whose initial distribution is the posterior $p_{\bX|\bY}$ and which evolves according to the forward SDE in Eq.~\eqref{eq.forward.sde}.  

Let $\bZ_t$ denote the forward process associated with the prior distribution, whose initial condition follows the Gaussian mixture model (GMM) given in Assumption~\ref{assump.gmm}, i.e.,
\[
p_{\bZ_0}(\bz_0) = \sum_{k=1}^K \pi_k \, \phi(\bz_0; \bmu_k, \bSigma),
\]
and evolves under the same forward SDE in Eq.~\eqref{eq.forward.sde}.  
From Eq.~\eqref{eq:zzt}, we can write
\[
p_{\bZ_t^{\bY}}(\bz_t, \by) = p_{\bZ_t | \bY}(\bz_t, \by),
\]
so that the target posterior score function is given by
\begin{equation}\label{eq.post.score.def.apdx}
S_{\bZ^{\bY}_t}(\bz_t | \by) 
= \frac{\nabla_{\bz_t} p_{\bZ_t|\bY}(\bz_t|\by)}{p_{\bZ_t|\bY}(\bz_t|\by)}.
\end{equation}

To compute $p_{\bZ_t|\bY}(\bz_t|\by)$, we apply the law of total probability over the latent variable $\xi$, which defines the  mixture probability of the GMM from Eq.~\eqref{eq.gmm.prior.def}, and we have
\begin{equation}\label{eq:expand-by-xi}
p_{\bZ_t|\bY}(\bz_t | \by) 
= \sum_{k=1}^K p_{\bZ_t|\bY,\xi}(\bz_t | \by, k)\, p_{\xi|\bY}(k | \by),
\end{equation}
where $p_{\xi|\bY}(k|\by) = \mathbb{P}(\xi=k | \bY=\by)$ denotes the conditional probability of drawing from the $k$-th Gaussian component given $\bY=\by$.  

Applying Bayes’ theorem to $p_{\bZ_t|\bY,\xi}(\bz_t | \by, k)$ in Eq.~\eqref{eq:expand-by-xi}, we can rewrite Eq.~\eqref{eq:expand-by-xi} as
\begin{align}
p_{\bZ_t|\bY}(\bz_t | \by)
&= \sum_{k=1}^K 
\frac{p_{\bY|\bZ_t,\xi}(\by|\bz_t,k)\, p_{\bZ_t|\xi}(\bz_t|k)}{p_{\bY|\xi}(\by|k)}\, p_{\xi|\bY}(k|\by) \\
&= \sum_{k=1}^K 
p_{\bY|\bZ_t,\xi}(\by|\bz_t,k)\, p_{\bZ_t|\xi}(\bz_t|k)\, \frac{p_{\xi|\bY}(k|\by)}{p_{\bY|\xi}(\by|k)} \\
&= \sum_{k=1}^K 
p_{\bY|\bZ_t,\xi}(\by|\bz_t,k)\, p_{\bZ_t|\xi}(\bz_t|k)\, \frac{p_{\xi}(k)}{p_{\bY}(\by)},\label{eq.pzt.y.expand1}
\end{align}
where we used the identity
\[
\frac{p_{\xi|\bY}(k|\by)}{p_{\bY|\xi}(\by|k)} = \frac{p_{\xi}(k)}{p_{\bY}(\by)}. 
\]

Since $\frac{p_{\xi}(k)}{p_{\bY}(\by)}$ is independent of $\bz_t$, taking the gradient of both sides of Eq.~\eqref{eq.pzt.y.expand1} with respect to $\bz_t$ gives
\begin{equation}\label{eq:posterior-pdf-grad}
\begin{aligned}
\nabla_{\bz_t} p_{\bZ_t|\bY}(\bz_t | \by)
= \sum_{k=1}^K 
\Big[\, 
&\nabla_{\bz_t} p_{\bY|\bZ_t,\xi}(\by|\bz_t,k)\, p_{\bZ_t|\xi}(\bz_t|k) \\
&+ p_{\bY|\bZ_t,\xi}(\by|\bz_t,k)\, \nabla_{\bz_t} p_{\bZ_t|\xi}(\bz_t|k)
\,\Big]
\frac{p_{\xi}(k)}{p_{\bY}(\by)}.
\end{aligned}
\end{equation}

Finally, substituting Eq.~\eqref{eq:posterior-pdf-grad} into the definition of the score function in Eq.~\eqref{eq.post.score.def.apdx} yields
\begin{equation}\label{eq:final-score-1}
S_{\bZ_t | \bY}(\bz_t | \by)
= 
\underbrace{\sum_{k=1}^K 
\frac{p_{\bY|\bZ_t,\xi}(\by|\bz_t,k)\, p_{\xi}(k)}
{p_{\bZ_t|\bY}(\bz_t|\by)\, p_{\bY}(\by)}\,
\nabla_{\bz_t} p_{\bZ_t|\xi}(\bz_t|k)}_{I_1}
+
\underbrace{\sum_{k=1}^K 
\frac{p_{\bZ_t|\xi}(\bz_t|k)\, p_{\xi}(k)}
{p_{\bZ_t|\bY}(\bz_t|\by)\, p_{\bY}(\by)}\,
\nabla_{\bz_t} p_{\bY|\bZ_t,\xi}(\by|\bz_t,k)}_{I_2}.
\end{equation}

We next simplify the two terms $I_1$ and $I_2$ in Eq.~\eqref{eq:final-score-1} to obtain a more compact and interpretable expression for the posterior score function.

\subsection{Simplification of $I_1$}
From Lemma~\ref{lemma.single.gaussian.trans}, the marginal prior density of the $k$-th Gaussian component at pseudo-time $t$ is given by
\[
p_{\bZ_t | \xi} (\bz_t | k) = \phi(\bz_t; \alpha_t \bmu_k,\, \beta_t^2 \bI_d + \alpha_t^2 \bSigma).
\]
Using the gradient of a Gaussian density, we obtain
\begin{equation}\label{eq:prior-cond-grad}
\nabla_{\bz} p_{\bZ_t | \xi} (\bz_t | k)
= p_{\bZ_t | \xi} (\bz_t | k)\, S_{\bZ_t|\xi}(\bz_t | k),
\end{equation}
where
\begin{equation}\label{eq:prior-score-gmm}
S_{\bZ_t|\xi}(\bz_t | k)
= -\big(\beta_t^2 \bI_d + \alpha_t^2 \bSigma \big)^{-1} \big(\bz_t - \alpha_t \bmu_k\big)
\end{equation}
is the score function of the Gaussian component $p_{\bZ_t | \xi} (\cdot | k)$.

Substituting Eq.~\eqref{eq:prior-cond-grad} into the expression of $I_1$ from Eq.~\eqref{eq:final-score-1}, we obtain
\begin{equation}\label{eq:I1-final}
\begin{aligned}
I_1
&= \sum_{k=1}^K 
\frac{p_{\bY | \bZ_t, \xi}(\by | \bz_t, k)\, p_{\xi}(k)}{p_{\bZ_t | \bY}(\bz_t | \by)\, p_{\bY}(\by)} 
\nabla_{\bz} p_{\bZ_t | \xi} (\bz_t | k) \\[4pt]
&= \sum_{k=1}^K 
\frac{p_{\bY | \bZ_t, \xi}(\by | \bz_t, k)\, p_{\bZ_t | \xi}(\bz_t | k)\, p_{\xi}(k)}
{p_{\bZ_t | \bY}(\bz_t | \by)\, p_{\bY}(\by)} 
S_{\bZ_t|\xi}(\bz_t | k) \\[4pt]
&= \sum_{k=1}^K 
\frac{p_{\xi, \bZ_t, \bY}(k, \bz_t, \by)}{p_{\bZ_t, \bY}(\bz_t, \by)}\, S_{\bZ_t|\xi}(\bz_t | k) \\[4pt]
&= \sum_{k=1}^K 
p_{\xi | \bZ_t, \bY}(k | \bz_t, \by)\, S_{\bZ_t|\xi}(\bz_t | k).
\end{aligned}
\end{equation}

Eq.~\eqref{eq:I1-final} represents the first component of the posterior score in Eq.~\eqref{eq.post.score.expression.final}, where each prior Gaussian component score function $S_{\bZ_t|\xi}(\bz_t | k)$ is weighted by its posterior responsibility $p_{\xi | \bZ_t, \bY}(k | \bz_t, \by)$.

\subsection{Simplification of $I_2$}\label{sec:I2}

To simplify $I_2$, we first handle the gradient term $\nabla_{\bz_t} p_{\bY | \bZ_t, \xi}(\by | \bz_t, k)$.  
By introducing $\bZ_0$ as an intermediate variable, the conditional density $p_{\bY | \bZ_t, \xi}(\by | \bz_t, k)$ can be expressed as
\begin{equation}\label{eq:y-zt-xi-decom}
\begin{aligned}
p_{\bY | \bZ_t, \xi}(\by | \bz_t, k)
&= \int p_{\bY | \bZ_0, \bZ_t, \xi}(\by | \bz_0, \bz_t, k)\, p_{\bZ_0 | \bZ_t, \xi}(\bz_0 | \bz_t, k)\, d \bz_0 \\
&= \int p_{\bY | \bZ_0}(\by | \bz_0)\, p_{\bZ_0 | \bZ_t, \xi}(\bz_0 | \bz_t, k)\, d \bz_0.
\end{aligned}
\end{equation}
Here, the second equality follows from the conditional independence  
$p_{\bY | \bZ_0, \bZ_t, \xi}(\by | \bz_0, \bz_t, k) = p_{\bY | \bZ_0}(\by | \bz_0)$,  
which holds because $\bY$ depends on $\bZ_t$ only through $\bZ_0$ as specified by the observation model in Eq.~\eqref{eq.filtering.obs}.  
The term $p_{\bY | \bZ_0}(\by | \cdot) \equiv p_{\bY | \bX}(\by | \cdot)$ denotes the likelihood function,  
and $p_{\bZ_0 | \bZ_t, \xi}(\cdot | \bz_t, k)$ represents the reverse transition kernel conditioned on the $k$-th Gaussian component of the prior.

Under the GMM assumption on $\bZ_0$, Proposition~\ref{prop.prior.trans.density} implies that the prior reverse transition kernel $p_{\bZ_0 | \bZ_t, \xi}(\bz_0 | \bz_t, k)$ in Eq.~\eqref{eq:y-zt-xi-decom} is Gaussian, given by
\begin{equation}\label{eq:reserse-kernel-gmm}
    p_{\bZ_0 | \bZ_t, \xi}(\bz_0| \bz_t, k) = \phi(\bz_0; \bmu_{0|t,k}(\bz_t), \bSigma_{0|t}),
\end{equation}
where the conditional mean $\bmu_{0|t,k}(\bz_t)$ and covariance $\bSigma_{0|t}$ are specified by Proposition~\ref{prop.prior.trans.density} as
\begin{equation}\label{eq:mu_sig_0tk}
\begin{aligned}
    \bmu_{0|t,k}(\bz_t) &= \bmu_k + \alpha_t \bSigma \left( \alpha_t^2 \bSigma + \beta_t^2 \bI \right)^{-1} (\bz_t - \alpha_t \bmu_k),\\
    \bSigma_{0|t} &= \bSigma - \alpha_t^2 \bSigma \left( \alpha_t^2 \bSigma + \beta_t^2 \bI \right)^{-1} \bSigma.
\end{aligned}
\end{equation}

Since the dependence of $p_{\bY | \bZ_t, \xi}(\by | \bz_t, k)$ on $\bz_t$ is solely through the conditional mean $\bmu_{0|t,k}(\bz_t)$ inside the Gaussian kernel 
$p_{\bZ_0 | \bZ_t, \xi}(\bz_0 | \bz_t, k)$ from Eq.~\eqref{eq:reserse-kernel-gmm}, 
we apply a reparameterization trick to make this dependency explicit:

\begin{equation}\label{eq:change-var}
    (\bZ_0 | \bZ_t, \xi=k) \deq \bmu_{0|t,k}(\bz_t) + \bL_{0|t} \bveps,
\end{equation}
where $\bveps \sim \mathcal{N}(\bzero, \bI_d)$ and $\bL_{0|t}$ is the Cholesky factor of $\bSigma_{0|t}$.
Using this reparameterization, the gradient term can be rewritten as
\begin{equation}\label{eq.i2.gradient}
    \begin{aligned}
    & \nabla_{\bz_t} p_{\bY | \bZ_t, \xi}(\by | \bz_t, k) \\
    &= \nabla_{\bz_t} \int p_{\bY | \bZ_0}(\by | \bz_0)\, 
       \phi(\bz_0; \bmu_{0|t,k}(\bz_t), \bSigma_{0|t})\, d\bz_0 \\
    &= \nabla_{\bz_t} 
       \mathbb{E}_{\bZ_0 \sim \mathcal{N}(\bmu_{0|t,k}(\bz_t),\, \bSigma_{0|t})}
       \left[p_{\bY | \bZ_0}(\by | \bZ_0)\right] \\
    &= \nabla_{\bz_t} 
       \mathbb{E}_{\bveps \sim \mathcal{N}(\bzero, \bI)}
       \left[p_{\bY | \bZ_0}\big(\by \,\big|\, 
       \bmu_{0|t,k}(\bz_t) + \bL_{0|t}\bveps \big)\right] \\
    &= \mathbb{E}_{\bveps \sim \mathcal{N}(\bzero, \bI)}
       \left[\nabla_{\bz_t} p_{\bY | \bZ_0}\big(\by \,\big|\, 
       \bmu_{0|t,k}(\bz_t) + \bL_{0|t}\bveps \big)\right] \\
    &= \bJ(t) \,
       \mathbb{E}_{\bveps \sim \mathcal{N}(\bzero, \bI)}
       \left[
       \bS_{\bY | \bZ_0}\big(\by \,\big|\, 
       \bmu_{0|t,k}(\bz_t) + \bL_{0|t}\bveps\big)\,
       p_{\bY | \bZ_0}\big(\by \,\big|\, 
       \bmu_{0|t,k}(\bz_t) + \bL_{0|t}\bveps\big)
       \right] \\
    &= \bJ(t)\,
       \mathbb{E}_{\bZ_0 \sim \mathcal{N}(\bmu_{0|t,k}(\bz_t),\, \bSigma_{0|t})}
       \left[
       \bS_{\bY | \bZ_0}(\by | \bZ_0)\,
       p_{\bY | \bZ_0}(\by | \bZ_0)
       \right] \\
    &= \bJ(t)\,
       \int \bS_{\bY | \bZ_0}(\by | \bz_0)\,
       p_{\bY | \bZ_0}(\by | \bz_0)\,
       p_{\bZ_0 | \bZ_t, \xi}(\bz_0 | \bz_t, k)\, d\bz_0,
    \end{aligned}
\end{equation}
where $\nabla p_{\bY | \bZ_0}(\by|\cdot) = p_{\bY | \bZ_0}(\by|\cdot)\, \bS_{\bY | \bZ_0}(\by|\cdot)$ follows from the Gaussian observation noise assumption, 
and $\bJ(t)$ is the Jacobian matrix of $\bmu_{0|t,k}(\cdot)$, given by
\begin{equation}\label{eq:mu-jacobian}
    \bJ(t) = 
    \alpha_t\, \bSigma \, (\alpha_t^2 \bSigma + \beta_t^2 \bI_d)^{-1}.
\end{equation}

Substituting Eq.~\eqref{eq.i2.gradient} into $I_2$ from Eq.~\eqref{eq:final-score-1}, we obtain
\begin{equation}\label{eq:I2-final}
\begin{aligned}
    I_2 
    &= \sum_{k=1}^K 
    \frac{ p_{\bZ_t | \xi}(\bz_t | k)\, p_{\xi}(k) }
         { p_{\bZ_t | \bY}(\bz_t | \by)\, p_{\bY}(\by) } 
       \nabla_{\bz_t} p_{\bY | \bZ_t, \xi}(\by | \bz_t, k) \\[4pt]
    &= \sum_{k=1}^K 
       \frac{ p_{\bZ_t, \xi}(\bz_t, k) }
            { p_{\bZ_t, \bY}(\bz_t, \by) } 
       \bJ(t) 
       \int \bS_{\bY | \bZ_0}(\by | \bz_0)\,
            p_{\bY | \bZ_0}(\by | \bz_0)\,
            p_{\bZ_0 | \bZ_t, \xi}(\bz_0 | \bz_t, k)\, d\bz_0 \\[4pt]
    &= \bJ(t)
       \sum_{k=1}^K 
       \int 
       \frac{
           p_{\bZ_t, \xi}(\bz_t, k)\,
           p_{\bY | \bZ_0}(\by | \bz_0)\,
           p_{\bZ_0 | \bZ_t, \xi}(\bz_0 | \bz_t, k)
       }
       { p_{\bZ_t, \bY}(\bz_t, \by) }\,
       \bS_{\bY | \bZ_0}(\by | \bz_0)\, d\bz_0 \\[4pt]
    &= \bJ(t)
       \sum_{k=1}^K 
       \int 
       \frac{ p_{\bZ_0, \xi, \bZ_t, \bY}(\bz_0, k, \bz_t, \by) }
            { p_{\bZ_t, \bY}(\bz_t, \by) }\,
       \bS_{\bY | \bZ_0}(\by | \bz_0)\, d\bz_0 \\[4pt]
    &= \bJ(t)
       \sum_{k=1}^K 
       \int p_{\bZ_0, \xi | \bZ_t, \bY}(\bz_0, k | \bz_t, \by)\,
       \bS_{\bY | \bZ_0}(\by | \bz_0)\, d\bz_0 \\[4pt]
    &= \bJ(t)
       \int p_{\bZ_0 | \bZ_t, \bY}(\bz_0 | \bz_t, \by)\,
       \bS_{\bY | \bZ_0}(\by | \bz_0)\, d\bz_0 \\[4pt]
    &= \bJ(t)\,
       \mathbb{E}_{\bZ_0 \sim p_{\bZ_0 | \bZ_t, \bY}}
       \big[\bS_{\bY | \bZ_0}(\by | \bZ_0)\big]\\[4pt]
    &= \bJ(t)\,
       \mathbb{E}
       \big[\bS_{\bY | \bX}(\by | \bZ_0) | \bZ_t = \bz_t, \bY = \by \big]
\end{aligned}
\end{equation}

The expression in Eq.~\eqref{eq:I2-final} represents the second term of the posterior score in Eq.~\eqref{eq.post.score.expression.final}, which corresponds to the conditional expectation of the observation likelihood score $\bS_{\bY | \bX}(\by | \cdot)$ with respect to the conditional density $p_{\bZ_0 \mid \bZ_t, \bY}$.

\section{Prior GM construction}\label{append.gm.prior}

In practical DA problems, only prior samples are typically available. To apply our posterior sampling framework, we must construct a prior GMM from these samples $\{\bx_k\}_{k=1}^K$. The prior covariance $\bSigma$ is first estimated as the sample covariance of $\{\bx_k\}_{k=1}^K$, with classical covariance localization applied as needed. For the GMM means, a naive choice would be to set $\bmu_k = \bx_k$ for each $k$, i.e., each prior sample defines a mixture component.

However, this naive construction inflates the variance of the resulting GMM. Since the samples $\{\bx_k\}_{k=1}^K$ already exhibit variance $\hat{\bSigma}$, assigning them as GMM means while also introducing a shared covariance $\bSigma$ effectively double-counts variability, producing a mixture that is more dispersed than the given prior ensemble. To correct this, we introduce a parameter $\gamma \in [0,1]$, referred to as the variance-splitting parameter, which balances the contribution of variance between the GMM covariance and the spread of the GMM means.

Formally, let $\Bar{\bSigma}$ denote the sample covariance from $\{\bx_k\}_{k=1}^K$, and let $\Bar{\bx} = \tfrac{1}{K}\sum_k \bx_k$ be the sample mean. We then construct the GMM with shared covariance and means
\begin{equation}
    \begin{aligned}
        \bmu_k &= \sqrt{1-\gamma^2} \, (\bx_k - \Bar{\bx}) + \Bar{\bx}, \\
        \bSigma &= \gamma \, \Bar{\bSigma}.
    \end{aligned}
\end{equation}

Let $\Bar{\bsigma}^2$ denote the sample variance of $\{\bx_k\}_{k=1}^K$. By construction, the variance of the resulting GMM satisfies
\begin{equation}
    \mathrm{Var}(\text{GMM}(\bmu_k, \bSigma)) 
    = \mathrm{Var}(\{\bmu_k\}) + \mathrm{diag}(\bSigma) 
    = (1-\gamma^2) \, \Bar{\bsigma}^2 + \gamma^2 \, \Bar{\bsigma}^2
    = \Bar{\bsigma}^2,
\end{equation}
ensuring consistency with the variance of the prior ensemble.

The value of $\gamma$ determines how the prior variability is represented in the GMM. A larger $\gamma$ emphasizes the global Gaussian covariance structure, while a smaller $\gamma$ retains more of the non-Gaussian features of the prior through the spread of $\bmu_k$. Figure~\ref{fig:gmm_prior_gamma} illustrates this effect: larger $\gamma$ values produce a more Gaussian-like prior, whereas smaller $\gamma$ values yield a particle-like distribution that better captures non-Gaussian features. Two extreme cases are particularly instructive. When $\gamma=0$, the covariance vanishes ($\bSigma=\bzero$) and the GMM reduces to a mixture of delta functions centered at the prior samples, closely resembling a particle filter. When $\gamma=1$, all means collapse to $\Bar{\bx}$, and the prior reduces to a single Gaussian $\mN(\Bar{\bx}, \Bar{\bSigma})$. However, it is worth noting that even when the prior is a single Gaussian, the posterior distribution can remain highly non-Gaussian, as will be demonstrated in Figure~\ref{fig:non-Gaussian-dist}. This is because the posterior samples are generated by solving the reverse SDE/ODE, without directly imposing any Gaussian constraints.

\begin{figure}[h]
    \centering
    \includegraphics[width=0.8\linewidth]{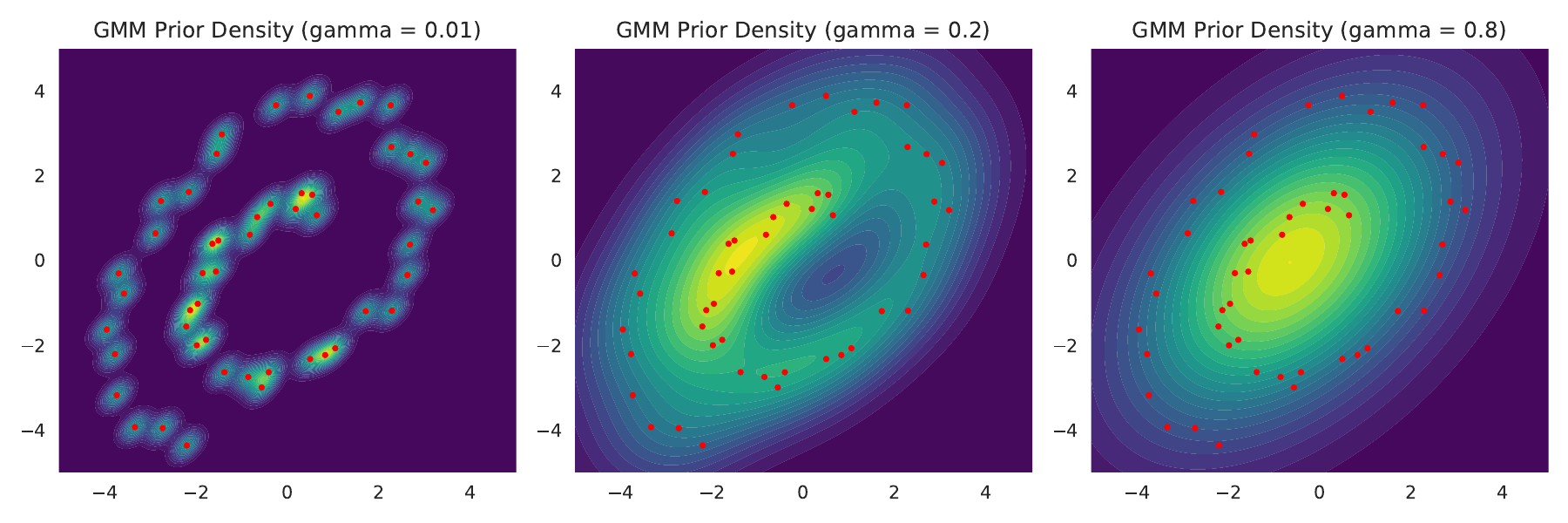}
    \caption{Constructed GMM priors for different values of $\gamma$. Red dots indicate the given prior ensemble members. As $\gamma$ increases, the prior approaches a Gaussian distribution, while smaller $\gamma$ values yield a particle-like distribution that preserves more non-Gaussian structure.}
    \label{fig:gmm_prior_gamma}
\end{figure}

\section{Alternative Approximations for $w^{\mathrm{obs}}(k,\by,\bz_t)$}\label{sec.w.obs.extra}
In this section, we present two computationally efficient approximations for $w^{\mathrm{obs}}(k,\by,\bz_t)$, derived from simplified linearizations of the nonlinear observation function $\mM(\cdot)$.  
These approximations reduce the cost of evaluating the Jacobian $\bJ_{\mM}(\bmu_{0|t,k}(\bz_t))$ for each mixture component $k$ when such computations are prohibitively expensive.

\begin{itemize}
    \item \textbf{Zeroth-order approximation.}  
    Neglect the linear term in the Taylor expansion of $\mM(\cdot)$ and approximate it as a constant function,
    $\mM(\bx) \approx \mM(\bmu_k^*)$, where $\bmu_k^*$ is the chosen evaluation point.  
    Under this approximation,
    \begin{equation}\label{eq:w-obs-simple1}
        w^{\mathrm{obs}}(k,\by,\bz_t) 
        = (\by - \mM(\bmu_k^*))^\top \bSigma_{\mathrm{obs}}^{-1} (\by - \mM(\bmu_k^*)) 
        = p_{\bY | \bX}(\by \mid \bmu_k^*).
    \end{equation}
    Hence, $w^{\mathrm{obs}}(k,\by,\bz_t)$ reduces to evaluating the likelihood function 
    $p_{\bY | \bX}(\by \mid \cdot)$ at the expansion point $\bmu_k^*$.  
    For instance, when $\bmu_k^* = \bmu_{0|t,k}(\bz_t)$, we have 
    $w^{\mathrm{obs}}(k,\by,\bz_t) = p_{\bY | \bX}(\by \mid \bmu_{0|t,k}(\bz_t))$.

    \item \textbf{Shared expansion point.}  
    Use a common expansion point $\bmu^*$ for all mixture components, i.e., $\bmu_k^* = \bmu^*$.  
    In this case, Eq.~\eqref{eq.w.obs.expression} simplifies to
    \begin{equation}\label{eq:w-obs-simple2}
        w^{\mathrm{obs}}(k,\by,\bz_t) 
        = (\by^* - \bH \bmu_{0|t,k}(\bz_t))^\top 
          \big(\bH \bSigma_{0|t} \bH^\top + \bSigma_{\mathrm{obs}}\big)^{-1} 
          (\by^* - \bH \bmu_{0|t,k}(\bz_t)),
    \end{equation}
    where $\by^* = \by - \mM(\bmu^*) + \bJ_{\mM}(\bmu^*)\bmu^*$ and $\bH = \bJ_{\mM}(\bmu^*)$.  
    This approach requires only a single evaluation of $\bJ_{\mM}(\cdot)$, significantly reducing computational cost.  
    Notably, this choice is also consistent with the linear observation model, for which $\bJ_{\mM}(\cdot)$ is constant.
\end{itemize}